\documentclass[11pt,a4paper]{amsart}
\usepackage{amssymb,xspace}
\usepackage{amstext}
\theoremstyle{plain}
\usepackage{amsbsy,amssymb,amsfonts,latexsym}
\usepackage[utf8]{inputenc}
\usepackage{xcolor}

\usepackage{enumerate}
\usepackage{graphics}

\date{\today}
\title{Multifractal phenomena and packing dimension}
\author{Fr\'ed\'eric Bayart, Yanick Heurteaux}
\address{
Laboratoire de Math\'ematiques\\
UMR 6620 - CNRS\\
Campus des C\'ezeaux\\
3, place Vasarely\\
TSA 60026\\
CS 60026\\
F-63178 Aubi\`ere cedex FRANCE\\
}
\email{Frederic.Bayart@math.univ-bpclermont.fr, Yanick.Heurteaux@math.univ-bpclermont.fr}



\newcommand{\veps}{\varepsilon}

\def\RR{\mathbb R}

\def\NN{\mathbb N}
\def\ZZ{\mathbb Z}
\def\TT{\mathbb T}

\def\CC{\mathbb C}

\def\card{\textrm{card}}
\def\dboxsup{\overline{\dim_{\rm{B}}}}
\def\dimh{\dim_{\mathcal H}}
\def\dimp{\dim_{\mathcal P}}

\DeclareMathOperator{\spe}{Sp}

\newtheorem{theorem}{Theorem}[section]

\newtheorem{lemma}[theorem]{Lemma}

\newtheorem{proposition}[theorem]{Proposition}

{\theoremstyle{definition}}
{\theoremstyle{definition}}

{\theoremstyle{definition}}

{\theoremstyle{definition}\newtheorem{definition}[theorem]{Definition}}

{\theoremstyle{definition}}

{\theoremstyle{definition}\newtheorem{remark}[theorem]{Remark}}

\newtheorem{question}[theorem]{Question}

\newtheorem*{thma}{Theorem A}
\newtheorem*{thmb}{Theorem B}
\newtheorem*{thmc}{Theorem C}

\begin{document}

\begin{abstract}
We undertake a general study of multifractal phenomena for functions. We show that the existence of several kinds of multifractal functions can be easily deduced from an abstract statement, leading to new results. This general approach does not work for Fourier or Dirichlet series. Using careful constructions, we extend our results to these cases.
\end{abstract}

\maketitle

\section{Introduction}

The starting points of this paper are the following three results on the multifractal properties of some classes of functions.

\smallskip
\noindent{$\blacktriangleright$ \bf Multifractal H\"older regularity of functions in Besov spaces.} Let $f:\RR^d\to\RR^d$ be locally bounded. We say that $f$ is H\"olderian with exponent $\alpha$ at $x_0$ (and we write $f\in \mathcal C^\alpha(x_0)$) if there exists $C,R>0$ and a polynomial $P$ of degree less than $\alpha$ such that 
\begin{equation}\label{eq:pointwiseholder2}
\left\Vert f(x)-P(x)\right\Vert_{L^\infty (B(x_0,r))}\le r^\alpha,\qquad\forall r\in(0,R].
\end{equation}

We define the lower pointwise H\"older exponent $h^-(x_0)$ of $f$ at $x_0$ as the supremum of the nonnegative real numbers $\alpha>0$ such that $f\in\mathcal C^{\alpha}(x_0)$ and the level sets $\mathcal E^{-}_{\rm HR}(h,f)$ and $E^{-}_{\rm HR}(h,f)$ (HR as H\"older regularity) :
\begin{eqnarray*}
\mathcal E^{-}_{\rm HR}(h,f)&=&\left\{x\in[0,1]^d;\ h^-(x)\leq h\right\}\\
E^{-}_{\rm HR}(h,f)&=&\left\{x\in[0,1]^d;\ h^-(x)=h\right\}.\\
\end{eqnarray*}

Let now $s>0$, $p\geq 1$ and consider the Besov space $B_{p,q}^s([0,1]^d)$ with $s-d/p>0$. For any function $f\in B_{p,q}^s([0,1]^d)$ and for almost all $x\in [0,1]^d$, $f$ is H\"olderian with exponent $\alpha$ at $x$ for any $\alpha<s$; however, it can happen that the regularity of $f$ at some point $x_0$ is worst but it cannot be too bad since $f$ always belong to $\mathcal C^{s-d/p}(x_0)$. The following theorem was obtained by St\'ephane Jaffard in \cite{Jaf00}.
\begin{thma}\ 
\begin{itemize}
\item[(i)] For all functions $f\in B_{p,q}^s([0,1]^d)$ and all $h\in[s-d/p,s]$, 
$$\dimh\big(\mathcal E^-_{\rm HR}(h,f)\big)\leq d+(h-s)p.$$
\item[(ii)] For quasi-all functions $f\in B_{p,q}^s([0,1]^d)$, for all  $h\in[s-d/p,s]$, 
$$\dimh\big(E^-_{\rm HR}(h,f)\big)=d+(h-s)p.$$
\end{itemize}
\end{thma}
The terminology quasi-all used here is relative to the Baire category theorem. It means that this property is shared by a residual set of functions in $B_{p,q}^s([0,1]^d)$. Moreover, $\dimh(E)$ means the Hausdorff dimension of $E$.

\smallskip
\noindent{$\blacktriangleright$\bf Multifractal divergence of Fourier series.} Let $f\in L^p(\TT)$, $p>1$, where $\TT=\RR/\ZZ$ is the unit circle, and let $(S_nf)_n$ be the sequence of the partial sums of its Fourier series. Carleson's Theorem says that the sequence $(S_nf(x))_n$ is bounded for almost every $x\in\TT$ and indeed is almost surely convergent to $f(x)$; however, it can happen that it is unbounded but not too badly: $S_nf(x)$ is always dominated by $n^{1/p}$. This motivates to introduce the lower divergence index $\beta_{\rm FS}^-(x_0)$ (FS as Fourier Series) at $x_0\in\TT$ as the infimum of the real numbers $\beta$ such that $S_n f(x_0)=O(n^\beta)$ and for $\beta\ge0$, the level sets 
\begin{eqnarray*}
\mathcal E^-_{\rm FS}(\beta,f)&=&\big\{x\in\TT;\ \beta^-_{\rm FS}(x)\geq \beta\big\}\\
E^-_{\rm FS}(\beta,f)&=&\big\{x\in\TT;\ \beta^-_{\rm FS}(x)=\beta\big\}\\
&=&\left\{x\in\TT;\ \limsup_{n\to+\infty}\frac{\log |S_nf(x)|}{\log n}=\beta\right\}.
\end{eqnarray*}
The following theorem summarizes results from \cite{Aub06,BAYHEUR1,BAYHEUR2}:
\begin{thmb}Let $p\geq 1$.
\begin{itemize}
\item[(i)] For all functions $f\in L^p(\TT)$ and all $\beta\in[0,1/p]$, 
$$\dimh\big(\mathcal E^-_{\rm FS}(\beta,f)\big)\leq 1-\beta p.$$
\item[(ii)] For quasi-all functions $f\in L^p(\TT)$, for all $\beta\in[0,1/p]$, 
$$\dimh\big(E^-_{\rm FS}(\beta,f)\big)=1-\beta p.$$
\end{itemize}
\end{thmb}

\smallskip
\noindent{$\blacktriangleright$\bf Multifractal radial behavior of harmonic functions.} Let $\mathcal S_d$ (resp. $B_{d+1}$) be the euclidean unit sphere (resp. unit ball) in $\RR^{d+1}$ and let $d\sigma$ be the normalized Lebesgue measure on $\mathcal S_d$. If $f\in L^1(\mathcal S_d)$, the Poisson integral $P[f]$ of $f$ is the harmonic function defined on $B_{d+1}$ by 
$$P[f](x)=\int_{S_d}P(x,\xi)f(\xi)d\sigma(\xi)$$
where $P(x,\xi)=\frac{1-\Vert x\Vert^2}{\Vert x-\xi\Vert^{d+1}}$ is the Poisson kernel. By Fatou's theorem, $(P[f](rx))_{r\in(0,1)}$ converges as $r\to 1$ for a.e. $x\in\mathcal S_d$. However, it can happen that it is unbounded but not too badly: $P[f](rx)$ is always dominated by $(1-r)^{-d}$.  This leads us to introduce the lower divergence index $\beta^-_{\rm HF}(x_0)$ (HF as Harmonic Functions) at $x_0\in\mathcal S_d$ as the infimum of the real numbers $\beta$ such that $P[f](rx_0)=O\big((1-r)^{-\beta}\big)$ and for $\beta\ge0$,  the level sets
\begin{eqnarray*}
\mathcal E^-_{\rm HF}(\beta,f)&=&\big\{x\in\mathcal S_d;\ \beta^-_{\rm HF}(x)\geq \beta\big\}\\
E^-_{\rm HF}(\beta,f)&=&\big\{x\in\mathcal S_d;\ \beta^-_{\rm HF}(x)=\beta\big\}\\
&=&\left\{x\in\mathcal S_d;\ \limsup_{r\to 1}\frac{\log |P[f](rx)|}{-\log (1-r)}=\beta\right\}.
\end{eqnarray*}
The following theorem was proved in \cite{BAYHEUR3}:
\vfill\eject
\begin{thmc}\ 
\begin{itemize}
\item[(i)] For all functions $f\in L^1(\mathcal S_d)$ and all $\beta\in[0,d]$, 
$$\dimh\big(\mathcal E_{\rm HF}^-(\beta,f)\big)\leq d-\beta.$$
\item[(ii)] For quasi-all functions $f\in L^1(\mathcal S_d)$, for all $\beta\in[0,d]$, 
$$\dimh\big(E_{\rm HF}^-(\beta,f)\big)=d-\beta.$$
\end{itemize}
\end{thmc}

\smallskip

\noindent $\blacktriangleright$ These three theorems share many similarities. This is clear if we look at their statement. This is also true if we look at their proofs: in each case, we begin by proving (i) using a certain type of maximal inequality, then  we build a saturating function verifying (ii) and finally we deduce residuality. Nevertheless, the technical details are rather different and involve arguments which seem very specific to the situation. Our first objective in this paper is to device a general framework to perform this process. We introduce multifractal analysis of sequences indexed by dyadic cubes. Such a framework already appeared in the literature (see for instance \cite{Jafal10}) but only to bound the Hausdorff dimension of level sets and not to construct multifractal functions. Our results will allow us to get new examples of multifractal phenomena (for instance, regarding the divergence of wavelet expansions) and in particular to obtain in a unified way Theorem A and C. It turns out that Theorem B does not fall into this general framework (although the level sets $E_{\rm FS}^-(\beta,f)$ may also be expressed using a sequence indexed by dyadic cubes). The main reason for that it the nonpositivity of the Dirichlet kernel. Thus, Theorem B requires supplementary arguments.

Our second aim is to investigate the existence of multifractal functions when we replace the lower divergence index by the upper divergence index. Namely, for the Fourier series case, define $\beta^+_{\rm FS}(x_0)$ as the supremum of the real numbers $\beta$ such that $n^\beta=O(S_n f(x_0))$ and for $\beta\ge 0$, the level sets
\begin{eqnarray*}
\mathcal E^+_{\rm FS}(\beta,f)&=&\big\{x\in\TT;\ \beta^+_{\rm FS}(x)\geq \beta\big\}=\left\{x\in\TT;\ \liminf_{n\to+\infty}\frac{\log |S_nf(x)|}{\log n}\geq \beta\right\}\\
E^+_{\rm FS}(\beta,f)&=&\big\{x\in\TT;\ \beta^+_{\rm FS}(x)=\beta\big\}=\left\{x\in\TT;\ \liminf_{n\to+\infty}\frac{\log |S_nf(x)|}{\log n}=\beta\right\}\\
E_{\rm FS}(\beta,f)&=&\big\{x\in\TT;\ \beta^+_{\rm FS}(x)=\beta^-_{\rm FS}(x)=\beta\big\}=\left\{x\in\TT;\ \lim_{n\to+\infty}\frac{\log |S_nf(x)|}{\log n}=\beta\right\}.\\
\end{eqnarray*}
It turns out that, when we investigate the size of $E_{\rm FS}^+(\beta,f)$, the pertinent notion is the packing dimension which is denoted by $\dimp$. We shall prove the following result.
\begin{theorem}\label{thm:mainfourier}Let $p\ge 1$.
\begin{itemize}
\item [(i)] For all functions $f\in L^p(\TT)$ and all $\beta\in [0,1/p]$, 
$$\dimp\big(\mathcal E^+_{\rm FS}(\beta,f)\big)\leq 1-\beta p.$$
\item [(ii)] There exists a function $f\in L^p(\TT)$ such that, for all $\beta\in [0,1/p]$, 
$$\dimh\big(E_{\rm FS}(\beta,f)\big)=\dimp\big(E_{\rm FS}(\beta,f)\big)=1-\beta p.$$
In particular, $\dimp\big(E^+_{\rm FS}(\beta,f)\big)=1-\beta p$.
\end{itemize}
\end{theorem}

We have a similar statement for the radial behavior of harmonic functions. Here, the upper divergence index $\beta^+_{\rm HF}(x_0)$ is defined as the supremum of the real numbers $\beta$ such that $(1-r)^{-\beta}=O\big(P[f](rx_0)\big)$ and for $\beta\ge0$, the level sets are
\begin{eqnarray*}
\mathcal E^+_{\rm HF}(\beta,f)&=&\big\{x\in\mathcal S^d;\ \beta^+_{\rm HF}(x)\geq \beta\big\}=\left\{x\in\mathcal S^d;\ \liminf_{n\to+\infty}\frac{\log |P[f](rx)|}{-\log (1-r)}\geq\beta\right\}\\
E^+_{\rm HF}(\beta,f)&=&\big\{x\in\mathcal S^d;\ \beta^+_{\rm HF}(x)=\beta\big\}=\left\{x\in\mathcal S^d;\ \liminf_{n\to+\infty}\frac{\log |P[f](rx)|}{-\log (1-r)}=\beta\right\}\\
E_{\rm HF}(\beta,f)&=&\big\{x\in\mathcal S^d;\ \beta^+_{\rm HF}(x)=\beta^-_{\rm HF}(x)=\beta\big\}=\left\{x\in\mathcal S^d;\ \lim_{n\to+\infty}\frac{\log |P[f](rx)|}{-\log (1-r)}=\beta\right\}.\\
\end{eqnarray*}
The analogue of Theorem C for the upper divergence index reads as follows.
\begin{theorem}\label{thm:mainpoisson}\ 
\begin{itemize}
\item [(i)] For all functions $f\in L^1(\mathcal S^d)$ and all $\beta\in [0,d]$, 
$$\dimp\big(\mathcal E^+_{\rm HF}(\beta,f)\big)\leq d-\beta.$$
\item [(ii)] There exists a nonnegative function $f\in L^1(\mathcal S^d)$ such that, for all $\beta\in [0,d]$, 
$$\dimh\big(E_{\rm HF}(\beta,f)\big)=\dimp\big(E_{\rm HF}(\beta,f)\big)=d-\beta.$$
In particular, $\dimp\big(E^+_{\rm HF}(\beta,f)\big)=d-\beta$.
\end{itemize}
\end{theorem}

As suggested above, Theorem \ref{thm:mainpoisson} will follow from our general framework as this will be the case for several other examples (like H\"older regularity or wavelet expansions). In comparison, the proof of Theorem \ref{thm:mainfourier} is much more difficult. 

Let us point out that the property required in part (ii) of Theorem \ref{thm:mainfourier} or \ref{thm:mainpoisson} is considerably stronger than the one needed in part (ii) of Theorem B or C. In particular, we will explain why there is no hope that the functions constructed in Theorem \ref{thm:mainfourier} or \ref{thm:mainpoisson} form a residual subset of the ambient space.

\medskip

The paper is organized as follows. In Section \ref{sec:gf}, we introduce our general framework and give the version of our main theorems in this context. Section \ref{sec:applications} is devoted to its applications in three different contexts: harmonic functions, H\"older regularity and wavelet expansions. In Section \ref{sec:fourier}, we prove Theorem \ref{thm:mainfourier} using partly the general framework, and partly a very careful construction. Finally, in Section \ref{sec:dirichlet}, we perform a multifractal analysis of the divergence of Dirichlet series. 

\medskip

We end up this introduction by some words on notations. We use classical notations and results for the different notions of dimensions we introduce. All the relative informations may be found in \cite{Fal03}.
 We will use the notation $f(x)\ll g(x)$ when there is some constant $C>0$ such that $|f(x)|\leq C|g(x)|$ for all (appropriate) $x$. If both $f(x)\ll g(x)$ and $g(x)\ll f(x)$ hold, we will write $f(x)\asymp g(x)$. If $I$ is a cube, then $c I$, $c>0$, will mean the cube with the same center and length multiplied by $c$.

\section{The general framework}\label{sec:gf}
The general framework our paper is based on is inspired by that exposed in \cite{Jafal10}. We start with a compact set  $K\subset\RR^N$, $N\geq 1$ and $d>0$.  We assume that, for each $j\geq 1$, there exists a family $\Lambda_j$ of subsets of $K$ such that
\begin{itemize}
\item for any $\lambda\in\Lambda_j$, there exists $x_\lambda\in K$ such that $B(x_\lambda,c\cdot 2^{-j})\subset \lambda\subset B(x_\lambda,C\cdot 2^{-j})$ (the constants $c$ and $C$ do not depend on $j$ and $\lambda$);
\item for any $\lambda\neq\mu\in\Lambda_j$, $\lambda\cap\mu=\emptyset$;
\item $\bigcup_{\lambda\in \Lambda_j}\lambda=K$;
\item $\textrm{card}(\Lambda_j)\asymp 2^{jd}$. 
\end{itemize}
With these assumptions at hand, it is clear that $\dimh(K)=\dimp(K)=d$ and that we may compute the dimension (either box, or Hausdorff, or packing) of any subset of $K$ using coverings or packings with  elements of $\Lambda=\bigcup_j \Lambda_j$. In most cases, $K$ will be equal to $[0,1]^d$ and $\Lambda_j$ will be the sequence of dyadic cubes of the $j$-th generation. The other example we will consider is $K=\mathcal S_d$. It is easy to figure out what can be the sequence $\Lambda$ in that case (see for example \cite{BAYHEUR3}). For convenience, we will call later $(\lambda)_{\lambda\in \Lambda_j}$ the family of dyadic cubes of the $j$-th generation and $|\lambda|$ will denote the diameter of $\lambda$. We will also assume that there exists in $K$ an increasing family $(F_\alpha)_{\alpha\in(0,d)}$ of subsets of $K$ such that, for any $\alpha\in(0,d)$, $\dimp(F_\alpha)=\dimh(F_\alpha)=\alpha$ and moreover $\mathcal H^{\alpha}(F_\alpha)>0$. Again this is true if $K=[0,1]^d$ or $K=\mathcal S^d$. If we omit the condition on the packing dimension, which is not necessary for the forthcoming Theorem \ref{thm:mainhausdorff}, all compact subsets of $\RR^N$ have this property (see \cite[Theorem 3.1]{BAYMFF}).

\smallskip

Let now $X\subset Y$ be two closed cones in a Banach space, namely $f+g\in X$ and $cf\in X$ for any $f,g\in X$ and any $c\geq 0$. The usual norm in $Y$ will be denoted by $\Vert \cdot \Vert$. We assume that for all $\lambda\in\Lambda$, we have a continuous map $f\in  Y\mapsto e_\lambda(f)$. We shall impose later conditions on these maps. Observe that, for any $x\in K$ and any $j\geq 0$, there exists a unique $\lambda\in\Lambda_j$ such that $x\in \lambda$. We will denote it by $I_j(x)$. We shall also denote by $e_j(f,x)$, or by $e_j(x)$ when there is no ambiguity, the value of $e_\lambda(f)$. 

\begin{definition}
Let $f\in Y$, $x_0\in  K$ and $\alpha\ge 0$. 
\begin{itemize}
\item[(i)] We say that $f\in\mathcal C^\alpha(x_0)$ if there exists $C\geq 0$ such that, for any $j$ large enough, $|e_j(f,x_0)|\leq C\cdot 2^{-\alpha j}$. The \emph{lower index} of $f$ at $x_0$ is then 
$$h^-(f,x_0)=\sup\big\{\alpha\geq 0;\ f\in\mathcal C^{\alpha}(x_0)\big\}$$
with the usual convention $\sup\left(\emptyset\right)=-\infty$.
\item[(ii)] We say that $f\in\mathcal I^\alpha(x_0)$ if there exists $C\geq 0$ such that, for any $j$ large enough, $|e_j(f,x_0)|\geq C\cdot 2^{-\alpha j}$. The \emph{upper index} of $f$ at $x_0$ is then 
$$h^+(f,x_0)=\inf\big\{\alpha\geq 0;\ f\in\mathcal I^{\alpha}(x_0)\big\}$$
with the usual convention $\inf\left(\emptyset\right)=+\infty$.
\end{itemize}
\end{definition}
Our first result is a bound on the dimension of the sublevel sets related to these indexes under mild conditions on $(e_\lambda)$. Let us introduce the following notation:
\begin{eqnarray*}
\mathcal F^-(\alpha,f)&=&\{x\in K;\ h^-(f,x)\leq \alpha\}\\
\mathcal F^+(\alpha,f)&=&\{x\in K;\ h^+(f,x)\leq \alpha\}.\\
\end{eqnarray*}

\begin{proposition}\label{prop:bounddimension}
Let $f\in Y$. Assume that there exist $p\geq 1$ and $C>0$ such that, for any $j\geq 0$, $\sum_{\lambda\in\Lambda_j}|e_\lambda(f)|^p\leq C$. Then for any $\alpha\in [0,d/p]$, 
$$\dimh\big(\mathcal F^-(\alpha,f)\big)\leq p\alpha\quad \mbox{and}\quad 
\dimp\big(\mathcal F^+(\alpha,f)\big)\leq p\alpha.$$
\end{proposition}
\begin{proof}
The result on Hausdorff dimension is already contained in \cite[Theorem 5]{Jafal10}. Let us recall the argument, in order to be self-contained. Let $n\ge 0$, $\veps>0$ and $x\in\mathcal F^-(\alpha,f)$. We can find $j\ge n$ such that
\begin{equation}\label{eq:dimh}
\vert e_j(f,x)\vert\ge 2^{-j(\alpha+\veps)}\asymp \vert I_j(x)\vert^{\alpha+\veps}.
\end{equation}
We can then find a covering $\mathcal R$ of $\mathcal F^-(\alpha,f)$ constituted of dyadic cubes of generation greater than $n$ and satisfying \eqref{eq:dimh}.
We have
$$\sum_{I\in\mathcal R}\vert I\vert^{(\alpha+2\veps)p}\ll\sum_{j\ge n}2^{-j\veps p}\sum_{\lambda\in\Lambda_j}\vert e_\lambda(f)\vert^p\ll 1$$
and we can easily conclude that
$$\dimh\left( \mathcal F^-(\alpha,f)\right)\le p(\alpha+2\veps)\ .$$

The same paper \cite{Jafal10} also states a result for the dimension of $\mathcal F^+(\alpha,f)$, but it is not about the dimension we are interested in (in \cite{Jafal10}, the lower packing dimension is considered whereas we are concerned with the usual packing dimension) and the condition made on $(e_\lambda)$ is different. Thus let us prove the statement about $\dimp\big(\mathcal F^+(\alpha,f)\big)$. Let $\veps>0$ and let $x\in\mathcal F^+(\alpha,f)$. Then there exists $j_x\geq 1$ such that, for any $j\geq j_x$, $|I_j(x)|^{\alpha+\veps}\leq |e_j(f,x)|$. Let 
$$\mathcal F_{l,\veps}^+(\alpha,f)=\left\{x\in\mathcal F^+(\alpha,f);\ \forall j\geq l,\ |I_j(x)|^{\alpha+\veps}\leq |e_j(f,x)|\right\},$$
so that 
$$\mathcal F^+(\alpha,f)=\bigcap_{n\geq 1}\bigcup_{l\geq 1}\mathcal F_{l,1/n}^+(\alpha,f).$$
By a classical property of the packing dimension (see \cite[Proposition 3.6]{Fal03}), it is sufficient  to prove that for any $l>0$ and any $\veps>0$, $\dboxsup\big(\mathcal F^+_{l,\veps}(\alpha,f)\big)\leq p(\alpha+\veps)$. For $j\geq l$, let $\Gamma_j$ be the set of dyadic cubes of the $j$-th generation meeting $\mathcal F_{l,\veps}^+(\alpha,f)$. Then for any $\lambda\in \Gamma_j$, $|\lambda|^{\alpha+\veps}\leq |e_\lambda(f)|$ so that 
$$\sum_{\lambda\in \Gamma_j}|\lambda|^{(\alpha+\veps)p}\leq \sum_{\lambda\in\Lambda_j}|e_\lambda(f)|^p\leq C.$$
Hence, the cardinal number of $\Gamma_j$ is dominated by $2^{jp(\alpha+\veps)}$, which yields the result.
\end{proof}

We now introduce the general framework required to obtain the existence of at least one multifractal function in a strong sense. Let us introduce, for $\alpha\geq 0$, the level sets
$$F(\alpha,f)=\left\{x\in K;\ h^-(f,x)=h^+(f,x)=\alpha\right\}.$$

\begin{theorem}\label{thm:mainpacking}
Assume that there exists $p\geq 1$ and a constant $C>0$ such that 
\begin{description}
\item[(GF1)] For any $f\in Y$, for any $j\geq 0$, $\left(\sum_{\lambda\in\Lambda_j}|e_\lambda(f)|^p\right)^{1/p}\le C \|f\|$;
\item[(GF2)] For any $j\geq 0$, for any $(a_\lambda)_{\lambda\in\Lambda_j}$, there exists $f\in X$ such that
\begin{itemize}
\item $\|f\|\le C\left(\sum_{\lambda\in\Lambda_j}|a_\lambda|^p\right)^{1/p}$;
\item $\forall \lambda\in\Lambda_j$, $e_\lambda(f)\geq\frac1C |a_\lambda|$;
\end{itemize}
\item[(GF3)] For any $f,g\in X$, for any $\lambda\in\Lambda$, for any $c\in\RR_+$, 
$$e_\lambda(cf)=ce_\lambda(f)$$
$$e_\lambda(f+g)\geq \max\big(e_{\lambda}(f),e_\lambda(g)\big).$$
\end{description}
Then there exists $f\in X$ such that, for all $\alpha\in[0,d/p]$, 
$$\dimp\big(F(\alpha,f)\big)=\dimh\big(F(\alpha,f)\big)=p\alpha.$$
\end{theorem}

\medskip

\begin{remark}\label{rque:mainpacking}
In particular, if we introduce
$$F^+(\alpha,f)=\left\{x\in K;\ h^+(f,x)=\alpha\right\},$$
we have for any $\alpha\in[0,d/p]$, 
\begin{equation}\label{eq:dimp}
\dimp\big(F^+(\alpha,f)\big)=p\alpha.
\end{equation}
This can be compared with the generic result obtained in Theorem \ref{thm:mainhausdorff}. Nevertheless, as it will be seen in Proposition \ref{prop:nonresiduality}, we cannot hope \eqref{eq:dimp} to be true for quasi-all functions in $Y$, even for a single $\alpha\in (0,d/p)$.
\end{remark}
Before to proceed with the proof, let us comment these three assumptions. Assumption {\bf (GF1)}, which comes from Proposition \ref{prop:bounddimension}, is useful to bound the dimension of the level sets. Observe also that it implies the following simple formulas
$$h^-(f,x)=\liminf_{j\to+\infty}\frac{\log |e_j(f,x)|}{-j\log 2}\quad\textrm{ and }\quad h^+(f,x)=\limsup_{j\to+\infty}\frac{\log |e_j(f,x)|}{-j\log 2}.$$
In particular,
$$F(\alpha,f)=\left\{x\in K\ ;\ \lim_{n\to+\infty}\frac{\log\vert e_j(f,x)\vert}{-j\log 2}=\alpha\right\}.$$
Assumption {\bf (GF2)} says that we can saturate the condition given by {\bf (GF1)}, whereas the last assumption {\bf (GF3)} is a property of regularity which is useful to add the singularities. Observe that it is satisfied when all maps $f\mapsto e_\lambda(f)$ are linear functionals and if $X$ is the ``positive'' cone $\left\{f\in Y\ ;\ \ e_\lambda(f)\geq 0,\ \forall\lambda\in\Lambda\right\}$, which will be the case of several of our examples. 

\smallskip

The proof of Theorem \ref{thm:mainpacking} is divided into a series of lemma, in which we will construct functions with stronger and stronger properties.
\begin{lemma}\label{lem:packing1}
Let $\alpha>0$ and $G\subset K$ be such that $\dboxsup(G)<p\alpha$. Then there exists $f\in X$ satisfying  $h^+(f,x)\leq \alpha$ for any $x\in G$.
\end{lemma}
\begin{proof}
Provided $j$ is large enough (say $j\ge j_0$), $G$ can be covered by at most $2^{jp\alpha}$ disjoint dyadic cubes in $\Lambda_j$. We denote by $\Gamma_j$ the set of these cubes. We then set $$a_\lambda=\left\{
\begin{array}{ll}
2^{-j\alpha}&\textrm{if }\lambda\in\Gamma_j\\
0&\textrm{otherwise}
\end{array}\right.$$
so that $\sum_{\lambda\in\Gamma_j}|a_\lambda|^p\leq 1$. By {\bf (GF2)}, we get a function $f_j\in X$ and it is enough to set $f=\sum_{j\geq j_0}j^{-2}f_j$ since, for any $x\in G$ and any $j\geq j_0$,
$$e_j(f,x)\geq j^{-2}e_j(f_j,x)\geq j^{-2}2^{-j\alpha}.$$
\end{proof}

\begin{remark}\label{rem:fd/p}
 If we use the inequality $\mbox{card}(\Lambda_j)\ll2^{jd}$ and if we define $a_\lambda=2^{-jd/p}$ when $\lambda\in\Lambda_j$, we can also construct in the same way a function $f\in X$ satisfying $h^+(f,x)\le d/p$ for any $x\in K$. This will be useful in the proof of Theorem \ref{thm:mainpacking}.
\end{remark}

\begin{lemma}\label{lem:packing2}
Let $\alpha\in(0,d/p)$ and $F\subset K$ be such that $\dimp(F)=p\alpha$. Then there exists $f\in X$ satisfying $h^+(f,x)\leq \alpha$ for any $x\in F$.
\end{lemma}
\begin{proof}
Let $(\alpha_l)$ be a sequence decreasing to $\alpha$. Then there exists a sequence $(G_{l,u})$ of subsets of $K$ such that $F\subset\bigcap_l\bigcup_u G_{l,u}$ and $\dboxsup\big(G_{l,u}\big)<p\alpha_l$ for all $l,u$. For each $l,u$, Lemma \ref{lem:packing1} gives us a function $f_{l,u}$ which we may assume to be normalized. We then set $f=\sum_{l,u}2^{-(l+u)}f_{l,u}$ and we pick $x\in F$. Then, for any $l\geq 1$, there exists $u\geq 1$ such that $x\in G_{l,u}$, so that, for any $j\geq 1$, 
$$e_j(f,x)\geq 2^{-(l+u)}e_j(f_{l,u},x).$$
Letting $j$ to $+\infty$ (independently of $l$ and $u$), we get
$$\limsup_{j\to+\infty}\frac {\log e_j(f,x)}{-j\log 2}\leq\limsup_{j\to+\infty}\frac {\log e_j(f_{l,u},x)}{-j\log 2}\leq \alpha_l.$$
Since this is true for any $l\geq 1$, this yields the result of Lemma \ref{lem:packing2}.
\end{proof}
\begin{proof}[Proof of Theorem \ref{thm:mainpacking}]
Let $(F_\alpha)_{\alpha\in(0,d/p)}$ be an increasing family of subsets of $K$ such that, for any $\alpha\in(0,d/p)$, $\dimp(F_\alpha)=\dimh(F_\alpha)=p\alpha$ and moreover $\mathcal H^{p\alpha}(F_\alpha)>0$ (see the hypothesis made on $K$ at the beginning of the section). Let $F_{d/p}=K$ and let $(\alpha_k)_{k\geq 1}$ be a dense sequence in $(0,d/p]$ containing $d/p$. For any $k\geq 1$, Lemma \ref{lem:packing2} (or Remark \ref{rem:fd/p}) yields the existence of a function $f_k$ associated to $F_{\alpha_k}$ which we may assume to be normalized.
We set $f=\sum_{k\geq 1}k^{-2}f_k$ and we consider $\alpha\in (0,d/p]$ and $x\in F_\alpha$. If $k\geq 1$ and $j\geq 1$, we have 
$$e_j(f,x)\geq k^{-2}e_j(f_k,x).$$
In particular, if $\alpha_k\geq\alpha$, $x\in F_{\alpha_k}$ and we get 
$$\limsup_{j\to+\infty}\frac{\log e_j(f,x)}{-j\log 2}\leq\alpha_k.$$
Since the sequence $(\alpha_k)_{k\ge 1}$ is dense and contains $d/p$, we thus deduce that $F_\alpha\subset \mathcal F^+(\alpha,f)$ for all $\alpha\in(0,d/p]$ and $\mathcal H^{p\alpha}(\mathcal F^+(\alpha,f))>0$. Observe now that  $\mathcal F^+(\alpha,f)\subset F(\alpha,f)\cup\bigcup_{\beta<\alpha}\mathcal F^-(\beta,f)$. Proposition \ref{prop:bounddimension} ensures that $\mathcal H^{p\alpha}\left(\bigcup_{\beta<\alpha}\mathcal F^-(\beta,f)\right)=0$ and we get $\mathcal H^{p\alpha}\big(F(\alpha,f)\big)>0$ which is sufficient to conclude when $\alpha>0$.

Les us finally remark that the equality $\dimp\big(F(0,f)\big)=\dimh\big(F(0,f)\big)=0$ is a consequence of Proposition \ref{prop:bounddimension}.
\end{proof}
\begin{remark}
We can slightly relax the assumptions {\bf (GF1)} and {\bf (GF2)}. In {\bf (GF1)}, it is sufficient that $\left(\sum_{\lambda\in\Lambda_j}|e_\lambda(f)|^q\right)^{1/q}\ll \|f\|$ for all $q>p$ and in {\bf (GF2)}, it is sufficient that, for any $r<p$, we can construct a function $f$ satisfying $\|f\|\ll \left(\sum_{\lambda\in\Lambda_j}|a_\lambda|^r\right)^{1/r}$.
\end{remark}

We now discuss the residuality of multifractal functions. This will only happen if we look at the lower index. We shall denote, for $\alpha\geq 0$, 
$$F^-(\alpha,f)=\big\{x\in K;\ h^-(f,x)=\alpha\big\}.$$
We need to add two supplementary conditions on the sequences $(e_\lambda(f))_{\lambda\in\Lambda}$. The first one is relative to the topological properties of these sequences, which is not surprising since we want a topological statement. Roughly speaking, {\bf(GF4)} says that the cone $Y$ contains a dense part $\mathcal D$ made with "regular functions". In particular, it is satisfied if $Y$ contains a dense part of fonctions $f$ satisfying $\vert e_\lambda(f)\vert\le C2^{-jd/p}$ for any $j$ and any $\lambda\in\Lambda_j$. Condition {\bf (GF5)} is another property of regularity, which is again satisfied if each map $f\mapsto e_\lambda(f)$ is a linear functional. Moreover, we assume that $Y$ is separable.

\begin{theorem}\label{thm:mainhausdorff}
Let $p\geq 1$ and assume that {\bf(GF1)}, {\bf(GF2)}, {\bf(GF3)} and the following properties are satisfied:
\begin{description}
\item[(GF4)] $Y$ is separable and there exists $\mathcal D\subset Y$ dense such that, for any $f\in\mathcal D$, 
$$\liminf_{j\to+\infty}\frac{\sup_{\lambda\in\Lambda_j}\log|e_\lambda(f)|}{-j\log 2}\geq \frac dp;$$
\item[(GF5)] For any $\lambda\in\Lambda$ and any $f,g\in Y$, $|e_{\lambda}(f+g)|\geq |e_\lambda(f)|-|e_\lambda(g)|$.
\end{description}
Then there exists a residual set $\mathcal R\subset Y$ such that, for any $f\in\mathcal R$ and any $\alpha\in[0,d/p]$, 
$$\dimh\big(F^-(\alpha,f)\big)=p\alpha.$$
\end{theorem}

As for Theorem \ref{thm:mainpacking}, this statement will be a consequence of the following proposition, which constructs a residual set for a fixed level.

\begin{proposition}\label{prop:hausdorff}
Under the assumptions of Theorem \ref{thm:mainhausdorff}, let $\alpha\in(0,d/p]$ and $E$ be a subset of $K$ such that $\mathcal H^{p\alpha}(E)<+\infty$. There exists a residual set $\mathcal R_E\subset Y$ such that, for any $g\in\mathcal R_E$ and any $x\in E$, $h^-(g,x)\leq \alpha$.
\end{proposition}

\begin{proof}
Without less of generality, we may assume that $\mathcal H^{p\alpha}(E)<1$. For any $n\geq 1$, one can find a finite set $\mathcal B_n$ of dyadic cubes of generation greater than $n$, covering $E$, and such that $\sum_{\lambda\in\mathcal B_n}|\lambda|^{p\alpha }\leq 1$. We denote by $\mathcal B_{n,j}$ the dyadic cubes of generation $j$ which are elements of $\mathcal B_n$. Let also $J_{n}\geq n$ such that $\mathcal B_{n,j}$ is empty provided $j\notin [n,J_n]$.
For a fixed $n\geq 1$ and a fixed $j\in\{n,\dots,J_n\}$, we may construct a function $f_{n,j}\in X$ such that, for any $\lambda\in\mathcal B_{n,j}$,  $e_\lambda(f_{n,j})\geq 2^{-j\alpha}$  and 
$$\|f_{n,j}\|^p\ll \sum_{\lambda\in\mathcal B_{n,j}}2^{-jp\alpha}\ll 1.$$
We then set $f=\sum_{n\geq 1}\sum_{j=n}^{J_n}n^{-2}j^{-2}f_{n,j}$ which belongs to $X$. Observe that if $\lambda\in\mathcal B_{n,j}$, then $e_\lambda(f)\geq j^{-4}2^{-j\alpha}$. Using {\bf (GF4)} and the separability of $Y$, we can  introduce  a dense sequence $(f_l)_{l\ge 1}\subset Y$ such that 
$$|e_\lambda(f_l)|\leq \frac{1}{2j^5}2^{-j\alpha}$$
for all dyadic cubes $\lambda\in\Lambda_j$, $j$ sufficiently large. We set $g_l=f_l+\frac 1lf$ so that the sequence $(g_l)_{l\ge 1}$ keeps being dense in $Y$. Let now $m\geq 1$ and $l\geq 1$ be fixed and let us set $N_{m,l}=\max(m,l)$.
For any $\lambda\in \mathcal B_{N_{m,l}}$ of the $j$-th generation, we have
\begin{eqnarray*}
e_{\lambda}(g_l)&\geq&\frac 1le_\lambda(f)-e_\lambda(f_l)\\
&\geq&\frac1{lj^4}2^{-j\alpha}-\frac{1}{2j^5}2^{-j\alpha}\\
&\geq&\frac{1}{2j^5}2^{-j\alpha}\quad\mbox{if }j\ge l.\\
\end{eqnarray*}
By continuity of the maps $g\mapsto e_\lambda(g)$ and because $\mathcal B_{N_{m,l}}$ is finite, there exists $\delta_{m,l}>0$ such that, for all $\lambda\in\mathcal B_{N_{m,l}}$, for all $g\in Y$ with $\|g-g_l\|<\delta_{m,l}$, then
$$e_\lambda(g)\geq \frac{1}{4j^5}2^{-j\alpha}.$$
We finally set $\mathcal R_E=\bigcap_{m\geq 1}\bigcup_{l\geq 1}B_Y(g_l,\delta_{m,l})$ which is a residual subset of $Y$ and let us consider $g\in\mathcal R_E$. For any $m\geq 1$, there exists $l\geq 1$ such that $\|g-g_l\|<\delta_{m,l}$. Let now $x\in E$. Then $x$ belongs to $\bigcup_{\lambda\in \mathcal B_{N_{m,l}} }\lambda$, so that
there exists $j\geq m$ satisfying $|e_j(g,x)|\geq \frac1{4j^5}2^{-j\alpha}$ which concludes the proof.
\end{proof}

\begin{remark}
To prove Proposition \ref{prop:hausdorff}, we do not need really {\bf (GF4)}. It suffices that there exists a dense set $\mathcal D\subset X$ such that, for all $f\in\mathcal D$, 
$$\liminf_{j\to+\infty}\frac{\sup_{\lambda\in\Lambda_j}\log|e_\lambda(f)|}{-j\log 2}>\alpha.$$
\end{remark}
\begin{proof}[Proof of Theorem \ref{thm:mainhausdorff}]
We argue as for Theorem \ref{thm:mainpacking}. Let $(F_\alpha)_{\alpha\in (0,d/p]}$ be an increasing family of subsets of $K$ such that $\dimh(F_\alpha)=p\alpha$ and $0<\mathcal H^{p\alpha}(F_\alpha)<+\infty$. Let $(\alpha_k)$ be a dense sequence in $(0,d/p]$ containing $d/p$ and let us consider the residual set $\mathcal R=\bigcap_k \mathcal R_{F_{\alpha_k}}$. For any $f\in\mathcal R$, any $\alpha\in(0,d/p]$ and any $x\in F_\alpha$, we easily get that $\liminf_{j\to+\infty}\frac{\log|e_j(f,x)|}{-j\log 2}\leq\alpha$ so that $F_\alpha\subset \mathcal F^-(\alpha,f)$ and $\mathcal H^{p\alpha}\left(\mathcal F^-(\alpha,f)\right)>0$. Writing $\mathcal F^-(\alpha,f)=F^-(\alpha,f)\cup\bigcup_{\beta<\alpha}\mathcal F^-(\beta,f)$ and using Proposition \ref{prop:bounddimension}, we get $\mathcal H^{p\alpha}\big(F^-(\alpha,f)\big)>0$.
\end{proof}

\begin{remark}
A remarkable feature of both the proofs of Theorems \ref{thm:mainpacking} and \ref{thm:mainhausdorff} is that they do not involve a particular increasing family of sets $(F_\alpha)$ satisfying $\dim(F_\alpha)=p\alpha$. In particular, the construction does not depend on the notion of well-approximable real numbers. In that sense, it is simpler than the previous constructions done in \cite{BAYHEUR1,BAYHEUR2,Jaf00}.
\end{remark}

\begin{remark}
The interest of having two cones $X$ and $Y$ becomes clearer after the proof of Proposition \ref{prop:hausdorff}. It is sometimes easier to build a saturating function in a cone (think again at the positive cone) and then to deduce residuality in the ambient space. 
\end{remark}


\section{Applications}\label{sec:applications}
\subsection{Harmonic functions}
We shall now prove that Theorem C and Theorem \ref{thm:mainpoisson} on harmonic functions follow directly from our general framework. The key point is to obtain a dyadic version of radial limits.  As a consequence of Harnack inequality, it will be easy to obtain for non-negative functions. Recall that $d\sigma$ is the normalized Lebesgue measure on $\mathcal S_d$.

\begin{lemma}\label{lem:poisson1}
Let $f\in L^1(\mathcal S_d)$, $f\geq 0$, $x\in\mathcal S_d$ and $j\geq 0$. Then for any $r\in \big[1-2^{-j},1-2^{-(j+1)}\big]$, 
$$P[f](rx)\asymp 2^{jd}\int_{\xi\in I_j(x)} P[f]\big((1-2^{-j})\xi\big)d\sigma(\xi).$$
\end{lemma}
\begin{proof}
This follows from Harnack inequality, which implies that there exist two constants $c,C>0$ such that, for any $f\in L^1(\mathcal S_d)$, $f\geq 0$, for any $x\in\mathcal S_d$, for any $j\geq 0$, for any $r\in \big[1-2^{-j},1-2^{-(j+1)}\big]$, for any $\xi\in I_j(x)$, 
$$c P[f]\big((1-2^{-j})\xi\big)\leq P[f](rx)\leq CP[f]\big((1-2^{-j})\xi\big).$$
Observe indeed that the distance from $(1-2^{-j})\xi$ to $rx$ is dominated by the distance of $rx$ to the boundary of $B_{d+1}$.
\end{proof}

This lemma leads us to define, for any $f\in L^1(\mathcal S_d)$, for any $j\geq 0$ and any $\lambda\in\Lambda_j$, 
$$e_\lambda(f)=\int_{\lambda}P[f]\big((1-2^{-j})\xi\big)d\sigma(\xi)$$
and to consider $X=Y=\{f\in L^1(\mathcal S_d);\ f\geq 0\}$. With these notations, Lemma \ref{lem:poisson1} yields, for any $f\in X$ and any $x\in\mathcal S_d$, 
\begin{eqnarray*}
\limsup_{r\to 0}\frac{\log|P[f](rx)|}{-\log(1-r)}&=&d-\liminf_{j\to+\infty}\frac{\log |e_j(f,x)|}{-j\log 2}\\
\liminf_{r\to 0}\frac{\log|P[f](rx)|}{-\log(1-r)}&=&d-\limsup_{j\to+\infty}\frac{\log |e_j(f,x)|}{-j\log 2}.
\end{eqnarray*}
Therefore, using the notations of the introduction and that of Section \ref{sec:gf}, we have 
$$\mathcal E^+_{\rm HF}(\beta,f)=\mathcal F^+(d-\beta,f)\quad \mbox{and}\quad E_{\rm HF}(\beta,f)=F(d-\beta,f).$$
Here, $p=1$ and we have to prove that the sequence $(e_\lambda)_{\lambda\in\Lambda}$ satisfies  the assumptions of the general framework. For {\bf (GF3)} and {\bf (GF5)}, this is trival. For {\bf (GF1)}, we just observe that, for any $j\geq 0$ and any $f\in L^1(\mathcal S_d)$, $f\geq 0$, 
$$\sum_{\lambda\in \Lambda_j}|e_\lambda(f)|= \int_{\mathcal S_d}P[f]\big((1-2^{-j})\xi\big)d\sigma(\xi)\leq \|f\|_1$$
since the Poisson kernel is a contraction on $L^1(\mathcal S_d)$. To prove {\bf (GF2)}, it suffices to set, for any $j\geq 0$ and any sequence $(a_\lambda)_{\lambda\in\Lambda_j}$, 
$$f=\sum_{\lambda\in\Lambda_j}2^{jd}|a_\lambda|\mathbf 1_{\lambda}.$$
It follows easily from Lemma 4 in \cite{BAYHEUR3} that, for any $\lambda\in \Lambda_j$, 
$$e_\lambda(f)=\int_{\lambda}P[f]\big((1-2^{-j})\xi\big)d\sigma(\xi)\ge\int_\lambda2^{jd}\vert a_\lambda\vert P[\mathbf 1_\lambda]\big((1-2^{-j})\xi\big)d\sigma(\xi)\gg |a_\lambda| .$$
Finally, {\bf (GF4)} is clear if we choose for $\mathcal D$ the set of nonnegative continuous functions on $\mathcal S_d$. For these functions, $\vert e_\lambda(f)\vert\le C2^{-jd}\Vert f\Vert_\infty$ if $\lambda\in\Lambda_j$.

Thus we may apply Theorem \ref{thm:mainpacking} and Theorem \ref{thm:mainhausdorff} to get Theorem C and Theorem \ref{thm:mainpoisson}, except that we control the dimension of the level sets only for nonnegative functions and that we have obtained residuality for nonnegative functions and not in $L^1(\mathcal S_d)$ (here $X=Y$ and indeed the behavior of $e_\lambda(f)$ characterizes the radial behaviour only for nonnegative functions). The first problem is easily tackled by observing that $|P[f](rx)|\leq P[|f|](rx)$ so that 
$\mathcal E_{\rm HF}^\pm(\beta,f)\subset\mathcal E_{\rm HF}^\pm(\beta,\vert f\vert)$ and conclusion (i) of Theorem C and Theorem \ref{thm:mainpoisson} remains true for any function  $f\in L^1(\mathcal S^d)$. To obtain conclusion (ii) of Theorem C, we need to modify slightly the proof of Proposition \ref{prop:hausdorff}. We take for $\mathcal D$ the set of continuous functions on $\mathcal S_d$, and the proof will work almost words for words if we observe that 
$$P[g_l](rx)\geq\frac 1lP[f](rx)-P[f_l](rx)$$
and that we may apply Lemma \ref{lem:poisson1} to $f$ which belongs to $X$.

\smallskip

We now include the example showing that we cannot expect to have residuality in Theorem \ref{thm:mainpacking}. The obstruction is very strong.

\begin{proposition}\label{prop:nonresiduality}
Quasi-all functions $f\in L^1(\mathcal S_d)$ satisfy $\beta^+_{HF}(x)\le0$ for all $x\in\mathcal S_d$.
\end{proposition}
\begin{proof}
Let $\veps>0$ and $\rho\in(0,1)$. We consider
$$\mathcal U(\veps,\rho)=\left\{f\in L^1(\mathcal S_d);\ \forall x\in\mathcal S_d,\ \exists r\in(\rho,1),\ |P[f](rx)|<\frac 1{(1-r)^\veps}\right\}$$
and we claim that $\mathcal U(\veps,\rho)$ is a dense open set. Indeed, it contains all continuous functions. Moreover, pick any $f\in \mathcal U(\veps,\rho)$. For any $x\in\mathcal S_d$, there exists $r_x\in (\rho,1)$ such that $|P[f](r_x x)|<\frac1{(1-r_x)^\veps}$. By continuity of $(g,y)\mapsto P[g](r_x y)$, there exists an open neighbourhood $\mathcal O_x$ of $x$ in $\mathcal S_d$ and a neighbourhood $\mathcal V_x$ of $f$ in $L^1(\mathcal S_d)$ such that
$$\forall g\in\mathcal V_x,\ \forall y\in \mathcal O_x,\ |P[g](r_xy)|<\frac1{(1-r_x)^\veps}.$$
By compactness, $\mathcal S_d$ is covered by a finite number of open sets $\mathcal O_x$, says $\mathcal O_{x_1},\dots,\mathcal O_{x_p}$. Then $\mathcal V_{x_1}\cap\dots\cap \mathcal V_{x_p}$ is a neighbourhood of $f$ contained in $\mathcal U(\veps,\rho)$. We now pick a sequence $(\veps_k)$ going to zero and a sequence $(\rho_l)$ going to 1 and observe that any function $f$ in the residual set $\bigcap_{k,l}\mathcal U(\veps_k,\rho_l)$ satisfies $\liminf_{r\to 1}\frac{\log |P[f](rx)|}{-\log(1-r)}\le 0$ and $\beta^+_{HF}(x)\le0$.
for all $x\in\mathcal S_d$.
\end{proof}

This proof can be easily adapted to the other examples of this paper. We only need the existence of a dense set of regular functions and the continuity of $(g,y)\mapsto P[g](ry)$ for a fixed value of $r$. We do not formulate a statement in our general context because we discretize the problem and we lose continuity at the boundary points of the dyadic cubes.


\subsection{Haar expansions}
Let $\Lambda$ be the set of dyadic intervals of $[0,1)$, $\varphi=\mathbf 1_{[0,1)}$ and $(\psi_\mu)_{\mu\in\Lambda}$ be the standard Haar functions with $L^\infty$-normalization. Recall that the sequence made of $\varphi$ and $(2^{j/2}\psi_\mu)_{\mu\in\Lambda}$ is an orthonormal basis in $L^2([0,1))$. If $j\geq 1$, define the $j$-th partial sum of the Haar expansion of $f$ by
$$T_{j}f=\langle f,\varphi\rangle\varphi+\sum_{i=0}^{j-1}\sum_{\mu\in\Lambda_i}2^{i/2}\langle f,\psi_\mu\rangle \psi_\mu.$$
Recall that $T_jf\ge 0$ if $f\ge 0$. As for Fourier series, $(T_{j}f(x))_{j\ge 1}$ converges almost everywhere to $f(x)$ and we have a control of the Hausdorff dimension of the sublevel sets of divergence. For any $\beta$, define
$$\mathcal E^{-}_{\rm HE}(\beta,f)=\left\{x\in [0,1);\ \limsup_{j\to+\infty}\frac{\log |T_j f(x)|}{j\log 2}\geq\beta\right\}.$$
Aubry has shown in \cite{Aub06} that if $\beta\in[0,1/2]$, $\dimh\big(\mathcal E^-_{\rm HE}(\beta,f)\big)\leq 1-2\beta$ and that, given $E\subset [0,1)$ with $\dimh(E)<1-2\beta$, there exists $f\in L^2([0,1))$ such that, for any $x\in E$, 
$$\limsup_{j\to+\infty}\frac{\log |T_j f(x)|}{j\log 2}\geq\beta.$$
Our general framework can be used to go much further. As usual, we also denote
\begin{eqnarray*}
E^{-}_{\rm HE}(\beta,f)&=&\left\{x\in [0,1);\ \limsup_{j\to+\infty}\frac{\log |T_j f(x)|}{j\log 2}=\beta\right\}\\
\mathcal E^{+}_{\rm HE}(\beta,f)&=&\left\{x\in [0,1);\ \liminf_{j\to+\infty}\frac{\log |T_j f(x)|}{j\log 2}\geq\beta\right\}\\
 E^{}_{\rm HE}(\beta,f)&=&\left\{x\in [0,1);\ \lim_{j\to+\infty}\frac{\log |T_j f(x)|}{j\log 2}=\beta\right\}.
\end{eqnarray*}

\begin{theorem}\label{thm:mainwavelet}
\ \begin{itemize}
\item[(i)] For all $\beta\in[0,1/2]$ and all $f\in L^2([0,1))$, 
$$\dimp\big(\mathcal E^+_{\rm HE}(\beta,f)\big)\leq 1-2\beta;$$
\item[(ii)] There exists a nonnegative function $f\in L^2([0,1))$ such that, for all $\beta\in[0,1/2]$, 
$$\dimh\big(E_{\rm HE}(\beta,f)\big)=\dimp\big(E_{\rm HE}(\beta,f)\big)=1-2\beta;$$
\item[(iii)] For quasi-all functions $f\in L^2([0,1))$, for all $\beta\in [0,1/2]$, 
$$\dimh\big(E^-_{\rm HE}(\beta,f)\big)=1-2\beta.$$
\end{itemize}
\end{theorem}

\begin{proof}
Let $f\in L^2([0,1))$. Recall that $T_jf$ is the orthogonal projection on the vector space  generated by $(\mathbf 1_\lambda)_{\lambda\in\Lambda_j}$. In particular, it is well known that $T_jf\ge 0$ when $f\ge 0$. Moreover, if $j\geq 1$ and $\lambda\in\Lambda_j$, then $T_j f$ is constant on $\lambda$. Thus we may set $e_\lambda(f)=2^{-j/2}T_jf(x)$ where $x$ is any element of $\lambda$ and it is easy to check that 
\begin{eqnarray*}
\liminf_{j\to+\infty}\frac{\log |T_jf(x)|}{j\log 2}&=&\frac 12-\limsup_{j\to+\infty}\frac{\log |e_j(f,x)|}{-j\log 2}\\
\limsup_{j\to+\infty}\frac{\log |T_jf(x)|}{j\log 2}&=&\frac 12-\liminf_{j\to+\infty}\frac{\log |e_j(f,x)|}{-j\log 2}.
\end{eqnarray*}
Therefore, Theorem \ref{thm:mainwavelet} will follow from the results of Section \ref{sec:gf} if we are able to prove that the sequence $(e_\lambda)_{\lambda\in\Lambda}$ satisfies the assumptions of the general framework with $p=2$. We set $X=\{f\in L^2([0,1));\ f\geq 0\}$ and $Y=L^2([0,1))$. As before, the verification of {\bf (GF3)} and {\bf (GF5)} are immediate, as soon as we observe that the $e_\lambda$ are linear and positive on $X$. In order to prove {\bf (GF1)}, we just observe that if $\lambda\in \Lambda_j$,
$\vert e_\lambda(f)\vert^2=\int_\lambda\left\vert T_jf(x)\right\vert^2dx$ so that
$$\sum_{\lambda\in\Lambda_j}\vert e_\lambda(f)\vert^2=\left\Vert T_jf\right\Vert_2^2\le\Vert f\Vert_2^2.$$
Property {\bf (GF4)} is easily obtained by taking for $\mathcal D$ the set of (finite) linear combination of elements of the Haar basis, that is the set of fonctions $f$ for which there exists $k\ge 0$ such that $f$ is constant on any $\lambda\in\Lambda_k$. Indeed, for such a function $f$,  $T_jf=f$ if $j\ge k$ and 
$$\vert e_\lambda(f)\vert\le2^{-j/2}\Vert T_jf\Vert_\infty=2^{-j/2}\Vert f\Vert_\infty\quad\mbox{if }\lambda\in\Lambda_j.$$ 
For the property of reconstruction {\bf (GF2)}, we start from $j\geq 1$ and a finite sequence $(a_\lambda)_{\lambda\in\Lambda_j}$. We then set
$$f=\sum_{\lambda\in\Lambda_{j}}2^{j/2}\vert a_\lambda\vert\mathbf 1_\lambda\in X.$$
Since $f$ is constant on any dyadic cube of the $j$-th generation, $T_j f=f$ and for any $\lambda\in\Lambda_j$, $e_\lambda(f)= |a_\lambda|$. Moreover,
$$\Vert f\Vert_2^2=\sum_{\lambda\in\Lambda_j}\int_\lambda\vert f(x)\vert^2dx=\sum_{\lambda\in\Lambda_j}\vert a_\lambda\vert^2=\sum_{\lambda\in\Lambda_j}\vert e_\lambda(f)\vert^2.$$
\end{proof}


\subsection{H\"older regularity}
We now show that Theorem A falls into our general framework and that we can also obtain results for pointwise anti-H\"olderian irregularity. We follow the definitions introduced in \cite{ClNi10} and \cite{ClNi11}. Let $f:\mathbb R^d\to\mathbb R$ be locally bounded and let $x_0\in\mathbb R^d$. The finite differences of arbitrary order of $f$ are defined inductively by
$$\Delta_h^1f(x_0)=f(x_0+h)-f(x),\quad\quad \Delta_h^{n+1}f(x_0)=\Delta_h^n f(x_0+h)-\Delta_h^n f(x_0).$$
It is known (see for example \cite{Clausel}) that if $\alpha>0$ is not an integer, then $f\in\mathcal C^\alpha(x_0)$ if and only if there exists $C,R>0$ such that,  
\begin{equation}\label{eq:pointwiseholder}
\sup_{\Vert u\Vert\leq r}\left\|\Delta_u^{[\alpha]+1}f\right\|_{L^\infty\big(B_u^\alpha(x_0,r)\big)}\leq Cr^\alpha,\qquad\forall r\in(0,R],
\end{equation}
where $B_u^\alpha(x_0,r)=\left\{x;\ \big[x;\ x+([\alpha]+1)u)\big]\subset B(x_0,r)\right\}.$
When $\alpha$ is an integer, there is an extra logarithmic term. However, this is unimportant if we look at the lower pointwise exponent $h^-(x_0)$ which could also be defined as the supremum of those $\alpha$ such that \eqref{eq:pointwiseholder} holds.

This motivates the following definition for anti-Hölderian irregularity. We say that $f\in\mathcal I^\alpha(x_0)$ if there exists $C,R>0$ such that, for any $r\in (0,R)$, 
$$\sup_{\Vert u\Vert\leq r}\left\|\Delta_u^{[\alpha]+1}f\right\|_{L^\infty\big(B_u^\alpha(x_0,r)\big)}\geq Cr^\alpha.$$
The upper pointwise H\"older exponent of $f$ at $x_0$ , denoted by $h^+(x_0)$, is the infimum of the real numbers $\alpha$ such that $f\in\mathcal I^{\alpha}(x_0)$. As usual, we introduce the level sets 
\begin{eqnarray*}
\mathcal E^+_{\rm HR}(h,f)&=&\left\{x\in[0,1]^d;\ h^+(x)\leq h\right\}\\
E_{\rm HR}(h,f)&=&\left\{x\in[0,1]^d;\ h^-(x)=h^+(x)=h\right\}\\.
\end{eqnarray*}
The following statement  adds informations to the results of Theorem A.
\begin{theorem}\label{main:holder}
\ 
\begin{itemize}
\item[(i)] For all functions $f\in B_{p,q}^s([0,1]^d)$ and all $h\in [s-d/p,s]$, 
$$\dimp\big(E_{\rm HR}(h,f)\big)\leq d+(h-s)p;$$
\item[(ii)] There exists a function $f\in B_{p,q}^s([0,1]^d)$ such that, for all $h\in [s-d/p,s]$, 
$$\dimp\big(E_{\rm HR}(h,f)\big)=\dimh\big(E_{\rm HR}(h,f)\big)=d+(h-s)p.$$
\end{itemize}
\end{theorem}

\begin{remark} Observe that point (i) in Theorem \ref{main:holder} is weaker than the analogue assertions in Proposition \ref{prop:bounddimension}, Theorem \ref{thm:mainfourier} and Theorem \ref{thm:mainpoisson}. We are not able to  get the stronger property 
$$\dimp\left(\mathcal E^+_{\rm HR}(h,f)\right)=\dimp\big(\{x\in [0,1]^d;\ h^+(x)\leq h\}\big)\leq d+(h-s)p$$ because of the weakness of  point (2) in Lemma  \ref{lem:holder1}. In particular, in this context, $h^+(x)$ is in general not equal to $\limsup_{j\to+\infty}\frac{\log d_j(f,x)}{-j\log 2}$ (see the definition of $d_j(f,x)$ below).
\end{remark}

We shall recall very briefly the basics of multiresolution wavelet analysis (for details, see for instance \cite{Mal98,Mey90}). For an arbitrary integer $N\geq s$, one can construct compactly supported functions $\varphi$ and $(\psi^{(i)})_{1\leq i<2^d}$ such that 
$$\{\varphi(x-k);\ k\in\ZZ^d\}\cup\{\psi^{(i)}(2^jx-k); 1\leq i<2^d,\ k\in\ZZ^d,\ j\in\ZZ\}$$ form an orthogonal basis of $L^2(\RR^d)$ (we choose the $L^\infty$-normalization of wavelets). We also assume that each $\psi^{(i)}$ has at least $N+1$ vanishing moments.
Then any function $f\in L^2([0,1]^d)$ can be decomposed as follows:
\begin{eqnarray*}
f(x)&=&c_0(f) \varphi(x)+\sum_{j=1}^{+\infty}\sum_{k\in \{0,\dots,2^j-1\}^d}\sum_{i=1}^{2^d-1}c_{j,k}^{(i)}(f)\psi^{(i)}(2^jx-k)\\
&=:&c_0(f) \varphi(x)+\sum_{\lambda\in \Lambda}\sum_{i=1}^{2^d-1}c_\lambda^{(i)}(f)\psi_\lambda^{(i)}(x),
\end{eqnarray*}
using the classical identification between the dyadic cube 
$$\lambda=\prod_{\ell=1}^d\left[\frac{k_\ell}{2^j},\frac{k_\ell+1}{2^j}\right)$$
and the numbers $j$ and $k=(k_1,\ldots,k_d)$.

It is well-known that the (global and local) regularity of a function is linked to the behavior of its wavelet coefficients. More precisely, we will need the notion of wavelet leaders. For every dyadic cube $\lambda\in\Lambda_j$, $j\geq 1$, one defines the wavelet leader $d_\lambda$ by setting
$$d_\lambda(f)=\sup\left\{|c_{\mu}^{(i)}(f)|;\ 1\leq i<2^d,\ \mu\subset 3\lambda\right\}.$$
As usual, if $\lambda$ is the unique cube in $\Lambda_j$ that contains the point $x$, the coefficient $d_\lambda(f)$ will be also denoted by $d_j(f,x)$.
The following results were obtained in \cite{ClNi10,ClNi11,Jaf04}.

\begin{lemma}\label{lem:holder1}
Let $f\in B_{p,q}^s(\RR^d)$ .
\begin{itemize}
\item[(1)] For any $x\in [0,1]^d$, $h^-(x)=\liminf_{j\to+\infty}\frac{\log d_j(f,x)}{-j\log 2}$;
\item[(2)] Let $x\in [0,1]^d$. If there exists $C>0$ such that for any $j\ge 0$, $d_j(f,x)\geq C2^{-j\alpha}$, then $f\in\mathcal I^\alpha(x)$;
\item[(3)] For any $x\in [0,1]^d$, $h^-(x)=h^+(x)=h$ if and only if $\lim_{j\to+\infty} \frac{\log d_j(f,x)}{-j\log 2}=h$.
\end{itemize}
\end{lemma}

In order to fit our general framework, we finally define $e_\lambda(f)=2^{\left(s-\frac dp\right)j}d_\lambda(f)$; we set $X$ the cone of $f\in B_{p,q}^s([0,1]^d)$ with nonnegative wavelet coefficients and $Y=B_{p,q}^s([0,1]^d)$.

\begin{lemma}\label{lem:holder2}
The sequence $(e_\lambda)_{\lambda\in\Lambda}$ satisfies properties {\bf(GF1)} to {\bf(GF5)}.
\end{lemma}

\begin{proof}
Let us recall that the $B_{p,q}^s([0,1]^d)$-norm of $f$ with wavelet coefficients $(c_\lambda^{(i)}(f))$ is
$$\|f\|_{B_{p,q}^s}=\left(\sum_{j\geq 1}\left(2^{(sp-d)j}\sum_{\lambda\in\Lambda_j}\sum_{i=1}^{2^d-1}|c_\lambda^{(i)}(f)|^p\right)^{q/p}\right)^{1/q}.$$
Using this formulation of the norm, we easily get {\bf(GF1)} and {\bf (GF2)}. For {\bf (GF1)}, this is done for instance in \cite{Seu16}. For {\bf (GF2)}, just define $f=\sum_{\lambda\in \Lambda_j}|a_\lambda|\psi_\lambda^{(1)}$. The verification of {\bf (GF3)} is easy (recall that we are working in the cone of functions with nonnegative wavelet coefficients), whereas we get {\bf (GF4)} by setting for $\mathcal D$ the set of functions with only a finite number of nonzero wavelet coefficients. Finally, for {\bf (GF5)}, observe that for $\lambda\in\Lambda$, if $\mu\subset 3\lambda$ and $1\leq i<2^d$ are such that $|c_\mu^{(i)}(f)|=d_\lambda(f)$, then
$$d_\lambda(f+g)\ge|c_\mu^{(i)}(f+g)|\geq |c_\mu^{(i)}(f)|-|c_\mu^{(i)}(g)|\geq d_\lambda(f)-d_\lambda(g).$$
\end{proof}

Let us now translate the results obtained in the general framework to the language of H\"older regularity. We have two kinds of indexes at our disposal: those coming from the general framework, which we shall denote by adding an index GF, like $h_{\rm GF}(x)$, and those coming from the H\"older exponent. Lemma \ref{lem:holder1} translates into 
\begin{eqnarray*}
(1)&&h^-(x)=h_{\rm GF}^-(x)+s-\frac dp\\
(2)&&h^+(x)\leq h_{\rm GF}^+(x)+s-\frac dp\\
(3)&&E_{\rm HR}(h,f)=F\left(h-s+\frac dp,f\right).
\end{eqnarray*}
Therefore, Theorems \ref{thm:mainhausdorff} and \ref{thm:mainpacking} exactly yield Theorem A and Theorem \ref{main:holder}.


\section{Fourier series}\label{sec:fourier}

We now investigate the multifractal analysis of the divergence of Fourier series. We first prove the inequality 
\begin{equation}\label{eq:efs}
\dimp\big(\mathcal E_{\rm FS}^+(\beta,f)\big)\leq 1-\beta p
\end{equation}
which is point (i) of Theorem \ref{thm:mainfourier}. It will be a consequence of the following localization lemma, which is a particular case of \cite[Lemma 2.5]{BAYHEUR2}. 
\begin{lemma}\label{lem:fourier1}
Let $p\ge 1$. There exists $\delta>0$ such that, for any $f\in L^p(\TT)$, for any $x\in\TT$ such that $|S_{2^j}f(x)|\geq \|S_{2^j}f\|_p$, 
$$\|S_{2^j}f\|_{L^p(3I_j(x))}\geq\frac{\delta2^{-j/p}}{j^{3/p}} |S_{2^j}f(x)|.$$
\end{lemma}

We begin by proving (\ref{eq:efs}) in the case $p>1$ (the case $p=1$ is a little more difficult). Let us introduce, for any dyadic interval  $\lambda\in \Lambda_j$, the quantity $e_\lambda(f)=\|S_{2^j}f\|_{L^p(3\lambda)}$. Let $0<\beta\le p$ (which is the only interesting case) and $x\in \mathcal E^+_{\rm FS}(\beta,f)$. For such an $x$, we have $\lim_{n\to+\infty}\vert S_nf(x)\vert=+\infty$. On the other hand, the Riesz theorem ensures that the sequence $\left(\left\Vert S_nf\right\Vert_p\right)_{n\ge 0}$ is bounded. We can then use Lemma \ref{lem:fourier1} when $n=2^j$ is sufficiently large. We get
\begin{eqnarray*}
\limsup_{j\to+\infty}\frac{\log |e_j(f,x)|}{-j\log 2}&\leq&\frac 1p-\liminf_{j\to+\infty}\frac{\log |S_{2^j}f(x)|}{j\log 2}\\
&\leq&\frac 1p-\liminf_{n\to+\infty}\frac{\log |S_nf(x)|}{\log n}.
\end{eqnarray*}
In other words and using our standard notations, we have shown that
\begin{eqnarray}
\label{eq:fourier}
\mathcal E^+_{\rm FS}(\beta,f)\subset \mathcal F^+\left(\frac 1p-\beta,f\right).
\end{eqnarray}
Proposition \ref{prop:bounddimension} gives immediately the conclusion, since Riesz theorem ensures that
$$\sum_{\lambda\in\Lambda_j}\|S_{2^j}f\|_{L^p(3\lambda)}^p\ll \|S_{2^j}f\|_{L^p(\TT)}^p\ll\|f\|^p.$$
In the case $p=1$, the Riesz inequality only says that 
$$\|S_n f\|_1\ll \log(n+1)\|f\|_1\ .$$
We first need to fix $\veps>0$ and to introduce  $e_\lambda(f)=2^{-j\veps}\|S_{2^j}f\|_{L^1(3\lambda)}$. It follows that
$$\sum_{\lambda\in\Lambda_j}|e_\lambda(f)|=2^{-j\veps}\sum_{\lambda\in\Lambda_j}\|S_{2^j}f\|_{L^1(3\lambda)}\ll 2^{-j\veps}\Vert S_{2^j}f\Vert_1\ll j2^{-j\veps}\Vert f\Vert_1\ll\Vert f\Vert_1,$$
which is the condition needed to apply Proposition \ref{prop:bounddimension}.

On the other hand, if $\beta>0$ and $x\in \mathcal E^+_{\rm FS}(\beta,f)$, it is always true that
$$\lim_{n\to+\infty}\frac{\vert S_nf(x)\vert}{\Vert S_nf\Vert_1}=+\infty,$$
so that we can use Lemma \ref{lem:fourier1}. We get 
$$\limsup_{j\to+\infty}\frac{\log |e_j(f,x)|}{-j\log 2}\le 1+\veps-\liminf_{n\to+\infty}\frac{\log |S_nf(x)|}{\log n}$$
and we can conclude that $\mathcal E_{\rm FS}^+(\beta,f)\subset\mathcal F^+(1+\veps-\beta,f)$. Proposition \ref{prop:bounddimension} then yields $\dimp\big(\mathcal E_{\rm FS}^+(\beta,f)\big)\leq 1+\veps-\beta$ and we let $\veps$ to 0.

\bigskip

Unfortunately, we cannot go much further staying inside the general framework. Indeed the Dirichlet kernel is not positive. In particular, it is harder than in the previous cases to add the singularities. The Dirichlet kernel is not a positive kernel but it is a real kernel and, in order to make $\vert S_nf(x)\vert$ large, it suffices to make large either its real part or its imaginary part. That is the way we will use to construct singularities. As usual, we will need to construct polynomials with spectra far from zero which take large values on big sets. The multiplication by $e^{inx}$ (which translates the spectrum) will not be in our situation a good idea, because the function $e^{inx}$ is not a real function. We will prefer to multiply by $\sin(2\pi nx)$ and to consider points $x$ where $\sin(2\pi nx)\ge c>0$. That is the reason why we first introduce the following compact set. 

\smallskip

Recall that any $x\in[0,1)$ can be uniquely written $\sum_{j\geq 1}\frac{\veps_j(x)}{2^{j}}$
with $\veps_j(x)\in\{0,1\}$ and with a sequence $(\veps_j(x))_{j\geq 1}$ that takes the value 0 infinitely often. 
Consider a sequence of positive integers $(m_k)_{k\ge 1}$ such that $m_{k+1}-m_k\geq 3$ for all $k$. 
For  $n\geq 1$, let
\begin{eqnarray*}
\Omega_n&=&\{1,\dots,n\}\backslash \bigcup_{k\geq 1}\{m_k,\ m_k+1,\ m_k+2\}
\end{eqnarray*}
and define
$$\Omega=\bigcup_{n\ge 1}\Omega_n.$$
We will assume that $(m_k)$ is sparse enough. In our context, this will mean that 
$$\textrm{card}\{k\in\NN;\ m_k\leq n\}=O(n^\gamma)$$
for some $\gamma\in (0,1)$. In particular, setting $u_n=\textrm{card}{(\Omega_n)}$ we have $u_n\sim n$. However, we will also need that $(m_k)$ is not too sparse and we also require that $m_{k+1}/m_k$ tends to 1. For instance, one may take $m_k=(k+1)^2$. 

Define 
$$
K=\overline{\big\{x\in [0,1);\ \veps_{m_k}(x)=\veps_{m_k+2}(x)=0\textrm{ and }\veps_{m_k+1}(x)=1\textrm{ for all }k\geq 1\big\}}.
$$
In other words, let $I_{\veps_1\cdots\veps_n}$ be the dyadic interval 
$$I_{\veps_1\cdots\veps_n}=\left[\sum_{j=1}^n \frac{\veps_j}{2^j}; \sum_{j=1}^n \frac{\veps_j}{2^j}+\frac1{2^n}\right).$$
Define the set $L_n$ of admissible words of length $n$ to be the words $\veps_1\cdots\veps_n$ with $\veps_j=0$ if $j\in\{m_k,\ k\ge 1\}$, $\veps_j=1$ if $j\in\{m_k+1,\ k\ge 1\}$ and $\veps_j=0$ if $j\in\{m_k+2,\ k\ge 1\}$. Then 
$$K=\bigcap_{n\ge 1}\bigcup_{\veps_1\cdots\veps_n\in L_n}\overline{I}_{\veps_1\cdots\veps_n}.$$
The choice of such a compact $K $ ensures the following.

\begin{lemma}\label{lem:sinmk}
For any $x\in K$ and for any $k\ge 1$, $\sin(2\pi 2^{m_k}x)\geq\frac{\sqrt 2}2$.
\end{lemma}

\begin{proof}
We just need to observe that, for any $x\in K$, 
$$2^{m_k}x\in \mathbb N+\left[\frac14,\frac38\right].$$
\end{proof}
Moreover, the choice of the sequence $(m_k)$ ensures that $K$ has dimension 1.

\begin{proposition}\label{prop:dimk}
 $\dimh(K)=\dimp(K)=1$.
 \end{proposition}
 
 \begin{proof}
 Of course, it suffices to prove that $\dimh(K)\ge 1$. Define the  probability measure $m$ on $K$ in the following way. Let $I\in \Lambda_n$ be a dyadic interval of the $n$-th generation. If 
$n+1\not\in\{m_k;\ k\ge 1\}\cup\{m_k+1;\ k\ge 1\}\cup\{m_k+2;\ k\ge 1\}$, 
 then $I$ is divided into two sons $I'$ and $I''$ of the $(n+1)$-th generation such that $m(I')=m(I'')=\frac12m(I)$. If $I=I_{\veps_1\cdots\veps_n}$ with $(n+1)\in\{m_k;\ k\ge 1\}$ then the total mass is transfered in $I_{\veps_1\cdots\veps_n010}$, that is 
 $m(I_{\veps_1\cdots\veps_n010})=m(I_{\veps_1\cdots\veps_n})$. The measure $m$ is clearly supported by the compact set $K$.
 
Moreover, for $n\geq 1$ and $I\in\Lambda_n$, we have $m(I)\le 2^{-u_n}$. Observe that $u_n\ge n-cn^\gamma$ so that
 $$m(I)\le 2^{-(n-cn^\gamma)}=\phi(\vert I\vert)$$
 where $\phi(t)=t\times 2^{c\left\vert \log_2(t)\right\vert^{\gamma}}$. In particular, $\mathcal H^\phi(K)>0$ and $\dimh(K)\ge 1$.
 \end{proof}
 
We now follow the way of the general case, with some additional complications, working first in sets with
small upper box dimensions. 
\begin{lemma}\label{LEMFOURIERBOX}
Define $M_k=2^{m_k-1}$ and $N_k=3\cdot 2^{m_k-1}$ if $k\ge 1$. 

Let $p\ge 1$, $s\in(0,1]$ and $G\subset K$ such that $\dboxsup(G)<s$. There exists $f\in L^p(\mathbb T)$, $\|f\|_p\leq 1$,  such that
\begin{itemize}
\item For any $x\in K$ and for any $k\ge 1$, 
$$\Re e(S_{N_k}f(x))\geq 0\quad and \quad\Im m(S_{N_k} f(x))\geq 0$$
\item For any $x\in[0,1)$ and for any $k\geq 1$, 
\begin{eqnarray*}
\Re e(S_n f(x))&=&\Re e(S_{N_{2k}}f(x))\textrm{ for any }n\in[N_{2k},M_{2k+2})\\
\Im m(S_n f(x))&=&\Im m(S_{N_{2k+1}}f(x))\textrm{ for any }n\in[N_{2k+1},M_{2k+3})\\
\end{eqnarray*}
\item For any $x\in G$,
\begin{eqnarray*}
\liminf_{k\to+\infty}\frac{\log\Re e(S_{N_{2k}}f(x))}{\log N_{2k}}&\geq& \frac{1-s}p\\
\liminf_{k\to+\infty}\frac{\log\Im m(S_{N_{2k+1}}f(x))}{\log N_{2k+1}}&\geq& \frac{1- s}p.
\end{eqnarray*}
\end{itemize}
\end{lemma}
\begin{proof}
Before to proceed with the details of the proof, let us comment the strategy. We will construct polynomials $P_k$ with small $L^p$-norms and which behave badly on $G$. Then, we will translate the Fourier spectrum of these polynomials, and add these translates, in order to get $f$. There are two problems to tackle. The first one is that, even if $P_k(x)$ is large, $S_n P_k(x)$ does not need to be large, and even positive, for small values of $n$. That is why we will alternate the addition of real polynomials and pure imaginary polynomials. The second trouble comes from the translation. As explained in the introduction of this section, in order to preserve the real part and the imaginary part of the polynomials $P_k$, we will multiply them by $\sin(2\pi nx)$, with $n$ correctly chosen. That is why we only consider points $x\in K$ for which we know that Lemma \ref{lem:sinmk} is valid.

\smallskip

Let us now proceed with the details. Since $\dboxsup(G)<s$, there exists $C_G>0$ such that, for any $k\geq 1$, we can find a set $\Gamma_k$ of closed dyadic intervals of length $2^{-(m_k-1)}$ 
which cover $G$, with $\card(\Gamma_k)\leq  C_G 2^{(m_k-1)s}$.
Let $J_k=\bigcup_{I\in\Gamma_k}I$ and $\chi_k$ the piecewise affine function defined by
$$\chi_k(x)=\max\left(0,1-2^{m_k-1}\mbox{dist}\,(x,J_k)\right).$$
Let us observe that $0\le\chi_k\le 1$, that $\chi_k(x)=1$ on each $I\in\Gamma_k$, that $\| \chi_k'\|_\infty\leq 2^{m_k-1}$
and that $\chi_k(x)=0$  on each closed dyadic interval $I$ of generation $m_k-1$ such that  $I\cap J_k=\emptyset$. In particular, 
$$\|\chi_k\|_p\leq \left(3C_G 2^{(m_k-1)s}\times2^{-(m_k-1)}\right)^{1/p}\ll 2^{(m_k-1)(s-1)/p}.$$
We can now approximate $2^{-(m_k-1)(s-1)/p}\chi_k$ (whose $L^p$-norm is controlled by a constant independent of $k$) by a trigonometric polynomial, using Fej\'er sums of order $2^{m_k-1}$. 
Applying for instance Lemma 1.6 of \cite{BAYHEUR1}, we get a polynomial $P_k$ satisfying 
$$\left\{
\begin{array}{rcl}
P_k(x)&\geq& 0\\
\|P_k\|_p&\ll& 1\\
\spe(P_k)&\subset&[-2^{m_k-1},2^{m_k-1}],\textrm{ where $\spe$ denotes the Fourier spectrum}\\
x\in G&\implies&P_k(x)\geq \frac14 2^{(m_k-1)(1-s)/p}.
\end{array}\right.$$
We then translate the Fourier spectrum of these polynomials by setting 
$$Q_k(x):=\sin(2\pi 2^{m_k}x)P_k(x).$$
Let us observe that the Fourier spectrum of $Q_k$ is contained in 
$$\left[-2^{m_k}-2^{m_k-1},-2^{m_k}+2^{m_k-1}\right]\cup \left[2^{m_k}-2^{m_k-1},2^{m_k}+2^{m_k-1}\right]=[-N_k,-M_k]\cup
[M_k,N_k].$$
Hence, since $m_{k+1}-m_k\geq 3$, these polynomials have disjoint spectra. Moreover, Lemma \ref{lem:sinmk} ensures that if $x\in K$, $Q_k(x)\ge \frac{\sqrt2}2 P_k(x)\ge 0$.
We finally define
$$f(x)=\sum_{j\geq 1}\frac{1}{j^2}Q_{2j}(x)+i\sum_{j\geq 1}\frac{1}{j^2}Q_{2j+1}(x)$$
which is a convergent series in $L^p(\TT)$ and we claim that $f$, up to renormalization, satisfies all the requirements of Lemma \ref{LEMFOURIERBOX}.
The first point is clear since $\Re e(S_{N_k}f)$ (resp. $\Im m(S_{N_k}f)$) is a nonnegative combination  of polynomials $Q_j$, which are nonnegative on $K$ by construction.
The second point follows directly from the construction and the description of the spectra of the polynomials $Q_j$ . For the last one, it suffices to observe that, for any $x\in G$, 
$$\Re e\big(S_{N_{2k}}f(x)\big)\geq \frac1{k^2}Q_{2k}(x)\geq \frac{\sqrt2}{2k^2}P_{2k}(x)\gg\frac{2^{m_{2k}(1-s)/p}}{k^2}.$$
Since $\log(N_{2k})$ behaves like $m_{2k}\log 2$, we get the result. We can do the same with $\Im m(S_{N_{2k+1}}f)$.
\end{proof}

\medskip

Following the strategy of the general framework, we can now go from subsets $G\subset K$ with controlled upper box dimension to subset $F\subset K$ satisfying $\dimp\left( F\right)=s$.

\begin{lemma}\label{LEMFOURIERPACKING}
Let  $s\in(0,1)$ and $F\subset K$ be such that $\dimp\left( F\right)=s$. There exists $f\in L^p(\mathbb T)$, $\|f\|_p\leq 1$,  such that
\begin{itemize}
\item For any $x\in K$ and for any $k\ge 1$, 
$$\Re e(S_{N_k}f(x))\geq 0\quad and\quad \Im m(S_{N_k} f(x))\geq 0$$
\item For any $x\in[0,1)$ and for any $k\geq 1$, 
\begin{eqnarray*}
\Re e(S_n f(x))&=&\Re e(S_{N_{2k}}f(x))\textrm{ for any }n\in[N_{2k},M_{2k+2})\\
\Im m(S_n f(x))&=&\Im m(S_{N_{2k+1}}f(x))\textrm{ for any }n\in[N_{2k+1},M_{2k+3})\\
\end{eqnarray*}
\item For any $x\in F$,
\begin{eqnarray*}
\liminf_{k\to+\infty}\frac{\log\Re e(S_{N_{2k}}f(x))}{\log N_{2k}}&\geq& \frac{1-s}p\\
\liminf_{k\to+\infty}\frac{\log\Im m(S_{N_{2k+1}}f(x))}{\log N_{2k+1}}&\geq& \frac{1- s}p.
\end{eqnarray*}
\end{itemize}
\end{lemma}
\begin{proof}
Let $(s_l)$ be a sequence decreasing to $s$. Then there exists a sequence $(G_{l,u})$ of subsets of $K$ such that $F\subset\bigcap_l\bigcup_u G_{l,u}$ and $\dboxsup(G_{l,u})<s_l$. For each couple $(l,u)$, Lemma \ref{LEMFOURIERBOX} gives us a function $f_{l,u}$. The function
$$f=\sum_{l,u\geq 1}\frac{1}{2^{l+u}}f_{l,u}$$
is the function we are looking for if we observe that if $x\in G_{l,u}$ and $k\ge 1$, 
$$\Re e(S_{N_{2k}}f(x))\ge\frac1{2^{l+u}}\Re e(S_{N_{2k}}f_{l,u}(x))\quad\mbox{and}\quad \Im m(S_{N_{2k+1}}f(x))\ge \frac1{2^{l+u}}\Im m(S_{N_{2k+1}}f_{l,u}(x)).$$
\end{proof}

\begin{remark}\label{rem:1+i} If $s=1$ and $F=K$ the function $f=\frac{1+i}{\sqrt2}$ obviously satisfies the conclusion of Lemma \ref{LEMFOURIERPACKING}.
\end{remark}
\medskip

To finish the proof of the construction of a function $f$ satisfying point (ii) of Theorem \ref{thm:mainfourier} we need now to construct a family $(F_\alpha)_{\alpha\in(0,1)}$ of increasing subsets of $K$ such that $\dimh(F_\alpha)=\dimp(F_\alpha)=\alpha$. This is not so easy because $K$ is not "selfsimilar" (the set $L_n$ of admissible words of length $n$ does not satisfy $\mbox{card}(L_n)\asymp 2^n)$. Nevertheless, we have the following.

\begin{lemma}\label{lem:falpha}
There exists a increasing family $(F_\alpha)_{0<\alpha<1}$  of subsets of $K$ such that for any $\alpha\in(0,1)$, $\dimh(F_\alpha)=\dimp(F_\alpha)=\alpha$ and $\mathcal H^{\psi_\alpha}(F_\alpha)>0$ for some gauge function $\psi_\alpha$ satisfying $\psi_\alpha(x)=o\left(x^s\right)$ for any $s<\alpha$.
\end{lemma}

Let us first suppose that Lemma \ref{lem:falpha} is true and finish the proof of Theorem \ref{thm:mainfourier}. Define $F_1=K$ and let $(\alpha_k)$ be a dense sequence in $(0,1]$ with  $\alpha_1=1$. Lemma \ref{LEMFOURIERPACKING} (or Remark \ref{rem:1+i}) gives us a function $f_k$ with $s=\alpha_k$.
As in the proof of Theorem \ref{thm:mainpacking}, we then set $f=\sum_{k\geq 1}\frac1{k^2}f_k$ and consider  $\alpha\in(0,1]$ and $x\in F_\alpha$. Let $k\ge 1$ such that  $\alpha_k\ge\alpha$ and $n\geq 1$ be very large. There exists
$q\geq 1$ such that $n\in [N_q,N_{q+1})$. 

If $q=2m$ is even, we write
\begin{eqnarray*}
 |S_n f(x)|&\geq&\Re e\big(S_nf(x)\big)\\
&=&\Re e\big(S_{N_q} f(x)\big)\\
&\geq&k^{-2}\times\Re e\big( S_{N_q}f_k(x)\big).
\end{eqnarray*}
If $q=2m+1$ is odd, we write
\begin{eqnarray*}
 |S_n f(x)|&\geq&\Im m\big(S_nf(x)\big)\\
&=&\Im m\big(S_{N_q} f(x)\big)\\
&\geq&k^{-2}\times\Im m\big(S_{N_q}f_k(x)\big).
\end{eqnarray*}
Now, $\log n\leq \log(N_{q+1})$ and $\frac{\log(N_{q})}{\log(N_{q+1})}\to 1$. Thus, 
$$\liminf_{n\to+\infty}\frac{\log|S_n f(x)|}{\log n}\geq \frac{1-\alpha_k}p.$$
Since $\alpha_k$ may be chosen arbitrarily close to $\alpha$, we get 
$$F_\alpha\subset \mathcal E^+_{\rm FS}(\beta,f)\quad\mbox{with}\quad\beta=\frac{1-\alpha}p.$$
In particular, $\mathcal H^{\psi_\alpha}\big(\mathcal E^+_{\rm FS}(\beta,f)\big)>0$. We now use the same argument as in Theorem \ref{thm:mainpacking}. Observe that 
$$\mathcal E^+_{\rm FS}(\beta,f)\subset E_{\rm FS}(\beta,f)\cup\bigcup_{\delta>\beta}\mathcal E^-_{\rm FS}(\delta,f)\ .$$
Theorem B ensures that $\mathcal H^{\psi_\alpha}\big(\bigcup_{\delta>\beta}\mathcal E^-_{\rm FS}(\delta,f)\big)=0$ and we get 
$$\mathcal H^{\psi_\alpha}\big(E_{\rm FS}(\beta,f)\big)>0\quad\mbox{and}\quad\dimh\big(E_{\rm FS}(\beta,f)\big)\ge\alpha=1-\beta p$$
which is sufficient to conclude when $\beta<1/p$. The case $\beta=1/p$ is a trivial consequence of point (i) of the theorem. 


\begin{question}Does there exists a real-valued function $f$ satisfying the conclusions of Theorem \ref{thm:mainfourier}?
\end{question}

We now finish this section with the proof of Lemma \ref{lem:falpha}. The sets $F_\alpha$ will be obtained as Besicovitch sets. Remember that if $0<\delta<1/2$ and if
$$E_\delta=\left\{x\in [0,1);\ \limsup_{n\to+\infty}\frac1n{\sum_{j=1}^{n} \veps_j(x)}\leq\delta \right\},$$
then, the Hausdorff dimension and the packing dimension of $E_\delta$ are equal to $\alpha(\delta)$, where $\alpha(\delta)$ denotes
$$\alpha(\delta):=-\delta\log_2(\delta)-(1-\delta)\log_2(\delta).$$
Moreover, $\mathcal H^{\alpha(\delta)}(E_\delta)=+\infty$ (see for instance \cite{Kaufman75} or \cite{MWW02}).

Observe now that the function $\delta\mapsto\alpha(\delta)$ is an increasing function mapping $(0,1/2)$ onto $(0,1)$. If $\alpha\in (0,1)$ and if $\delta$ is the unique real number in $(0,1/2)$ such that $ \alpha=\alpha(\delta)$ we set $F_\alpha=E_\delta\cap K$ and we claim that the family $(F_\alpha)_{0<\alpha<1}$ satisfies the conclusion of Lemma \ref{lem:falpha}. This will be due to the fact that the sequence $(m_k)$ is sparse.

\medskip

It is clear that $\dimp(F_\alpha)\le\dimp(E_\delta)=\alpha(\delta)=\alpha$. So we just have to  prove that  $\mathcal H^{\psi_\alpha}(F_\alpha)>0$ for some suitable gauge function $\psi_\alpha$.
As in the proof of Proposition \ref{prop:dimk}, we use an appropriate probability measure $m$ which is supported by $E_\delta\cap K$.

We define $m$ on the dyadic intervals by the following formulas. 
\begin{itemize}
\item If $n \in \Omega$, then
$$m(I_{\veps_1\cdots\veps_n})=
\left\{ 
\begin{array}{ll}
\delta m(I_{\veps_1\cdots\veps_{n-1}})&\textrm{if }\veps_n=1\\
(1-\delta)  m(I_{\veps_1\cdots\veps_{n-1}})&\textrm{otherwise}.
\end{array}\right.$$
\item If $n\in\{m_k;\ k\ge 1\}$ or $n\in\{m_k+2;\ k\ge 1\}$, then $m(I_{\veps_1\cdots\veps_{n-1}0})=m(I_{\veps_1\cdots\veps_{n-1}})$ and $m(I_{\veps_1\cdots\veps_{n-1}1})=0$
\item If $n\in\{m_k+1;\ k\ge 1\}$, then $m(I_{\veps_1\cdots\veps_{n-1}1})=m(I_{\veps_1\cdots\veps_{n-1}})$ and $m(I_{\veps_1\cdots\veps_{n-1}0})=0$.
\end{itemize}
The last two parts of the definition ensure that $m(K)=1$. 

It is well known that, with respect to the measure $m$,  the sequence $(\veps_j)_{j\in \Omega}$ is a sequence of independent Bernoulli variables with $m(\veps_j=1)=\delta$, $j\in \Omega$. Moreover, if $x\in K$, and if $I_n(x)$ is the unique dyadic interval of the $n$-th generation that contains $x$, we have 
$$m\left(I_n(x)\right)=\prod_{j\in\Omega_n}\delta^{\veps_j(x)}(1-\delta)^{1-\veps_j(x)}.$$
Define
$$X_j(x)=-\veps_j(x)\log_2(\delta)-(1-\veps_j(x))\log_2(1-\delta).$$
The random variables $(X_j)_{j\in\Omega}$ are independent uniformly distributed with $E(X_j)=\alpha(\delta)=\alpha$ and we have for any $x\in K$
$$-\log_2m\left(I_n(x)\right)=\sum_{j\in\Omega_n}X_j(x):=S_n(x).$$
Let $\sigma^2=V(X_j)$ and recall that the cardinal number of $\Omega_n$ is $u_n$. The law of the iterated logarithm ensures that $dm$-almost surely,
\begin{equation}\label{eq:loglog}
\liminf_{n\to+\infty} \frac{S_n-u_n\alpha}{\sqrt{2u_n\log\log u_n}}=-\sigma.
\end{equation}
In particular, considering $\theta\in (0,1)$ such that $\theta>1/2$, $dm$-almost surely, there exists $n_0\in\NN$ such that, for any $n\ge n_0$, $S_n\geq u_n\alpha-u_n^{\theta}$, namely
\begin{equation}\label{eq:loglog2}
m(|I_n(x)|)\leq 2^{-u_n\alpha+u_n^{\theta}}.
\end{equation}
Using the inequalities $n\geq u_n\geq n-c n^\gamma$, this in turn yields
$$m(|I_n(x)|)\leq |I_n(x)|^{\alpha} \exp\left(d\left\vert\log |I_n(x)|\right\vert^{\theta}\right)$$
for some positive constant $d$ if $\theta$ also satisfies $\theta>\gamma$.

Let $K_\delta$ be the set of points of $K$ that satisfy \eqref{eq:loglog} and let   $\psi_\alpha(t)=t^\alpha\exp\left(d\left\vert\log t\right\vert^{\theta}\right)$. Equation \eqref{eq:loglog2} and the mass distribution principle ensure that $\mathcal H^{\psi_\alpha}(K_\delta)>0$. 

On the other hand, if $x\in K_\delta$, we have 
$$\lim_{n\to+\infty}\frac1{u_n}\sum_{j\in\Omega_n}\veps_j(x)=\delta.$$
Observing that $\sum_{j=1}^n\veps_j(x)=\sum_{j\in\Omega_n}\veps_j(x)+O\left(n^{\gamma}\right)$ and that $u_n\sim n$, we can conclude that
$$\lim_{n\to+\infty}\frac1{n}\sum_{j=1}^n\veps_j(x)=\delta.$$
Finally, $K_\delta\subset E_\delta$ and $\mathcal H^{\psi_\alpha}\left(F_\alpha\right)=\mathcal H^{\psi_\alpha}\left(E_\delta\cap K\right)\ge\mathcal H^{\psi_\alpha}\left(K_\delta\right)>0$.


\section{Dirichlet series}\label{sec:dirichlet}
\subsection{Introduction} To end up this paper, we now deal with the multifractal analysis of the divergence of Dirichlet series. Let us introduce the Hardy space of Dirichlet series
$$\mathcal H^2=\left\{g(s)=\sum_{k\geq 1}a_k k^{-s};\ \|g\|_{\mathcal H^2}^2=\sum_{k\geq 1}|a_k|^2<+\infty\right\}.$$
A Dirichlet series $g(s)=\sum_{k\geq 1}a_k k^{-s}$ in $\mathcal H^2$ defines an holomorphic function in the half-plane $\mathbb C_{1/2}=\left\{s\in\CC;\ \Re e(s)>\frac12\right\}$. On the line $\Re e(s)=\frac 12$, the Dirichlet series may diverge but the Cauchy-Schwarz inequality ensures that 
$$\left|\sum_{k=1}^na_k k^{-\frac 12+it}\right|\ll \log(n+1)^{1/2}\|g\|_{\mathcal H^2}$$
for all $t\in\RR$ and all $n\in\NN^*$. Moreover, a version of Carleson's convergence theorem has been obtained in \cite{HS} and in \cite{KONQUEFF}: for almost all $t\in\RR$, the series $\sum_{k\geq 1}a_kk^{-\frac 12+it}$ converges. Define, for $t\in\RR$, 
\begin{eqnarray*}
\beta^-_{\rm DS}(t)&=&\limsup_{n\to+\infty}\frac{\log \left|\sum_{k=1}^na_k k^{-\frac 12+it}\right|}{\log\log n}\ =\ \inf\left(\left\{\beta\in\RR\ ;\ \left|\sum_{k=1}^na_k k^{-\frac 12+it}\right|\ll\log(n+1)^\beta\right\}\right)\\
\beta^+_{\rm DS}(t)&=&\liminf_{n\to+\infty}\frac{\log \left|\sum_{k=1}^na_k k^{-\frac 12+it}\right|}{\log\log n}\ =\ \sup\left(\left\{\beta\in\RR\ ;\ \left|\sum_{k=1}^na_k k^{-\frac 12+it}\right|\gg\log(n+1)^\beta\right\}\right)
\end{eqnarray*}
and consider, for $\beta\in[0,1/2]$, the associated sets $\mathcal E^-_{\rm DS}(\beta,g)$, $\mathcal E^+_{\rm DS}(\beta,g)$, $E^-_{\rm DS}(\beta,g)$, $E^+_{\rm DS}(\beta,g)$, $E_{\rm DS}(\beta,g)$. 
\begin{theorem}\label{thm:maindirichlet}
\ \begin{itemize}
\item[(i)] For all $g\in\mathcal H^2$ and all $\beta\in[0,1/2]$, 
$$\dimh\big(\mathcal E^-_{\rm DS}(\beta,g)\big)\leq 1-2\beta\quad \mbox{and}\quad \dimp\big(\mathcal E^+_{\rm DS}(\beta,g)\big)\leq 1-2\beta;$$
\item[(ii)] For quasi-all functions $g\in\mathcal H^2$, for all $\beta\in[0,1/2]$, 
$$\dimh\big(E^-_{\rm DS}(\beta,g)\big)= 1-2\beta;$$
\item[(iii)]There exists a function $g\in\mathcal H^2$ such that, for all $\beta\in[0,1/2]$, 
$$\dimh\big( E_{\rm DS}(\beta,g)\big)=\dimp\big( E_{\rm DS}(\beta,g)\big)= 1-2\beta.$$
\end{itemize}
\end{theorem}

Unfortunately, the study of the pointwise behavior for Dirichlet series is more difficult than for Fourier series (for instance, we do not have an analogue to Fej\'er's kernel) and is not a consequence of our general context (for more information on the problems arising in the theory of Hardy spaces of Dirichlet series, the reader may consult \cite{QuQu13} or \cite{Qu15}). Our strategy to prove Theorem \ref{thm:maindirichlet} is to compare the behavior of Dirichlet series with that of Fourier integrals of $L^2$-functions by a discretization argument (this idea comes from \cite{KONQUEFF}). An intermediate step will be to prove a corresponding result for Fourier integrals. The construction of multifractal functions will use the work done in Section \ref{sec:fourier}.

Let us introduce some terminology. Let $F\in L^2([0,+\infty))$. The Cauchy-Schwarz inequality ensures that for any $R>0$ and any $t\in\RR$, 
$\left|\int_0^R F(u)e^{itu}du\right|\leq R^{1/2}\|F\|_2$. Define
\begin{eqnarray*}
\beta^-_{\rm FI}(t)&=&\limsup_{R\to+\infty}\frac{\log \left|\int_0^R F(u)e^{itu}du\right|}{\log R}\\
\beta^+_{\rm FI}(t)&=&\liminf_{R\to+\infty}\frac{\log \left|\int_0^R F(u)e^{itu}du\right|}{\log R}
\end{eqnarray*}
and the corresponding sets $\mathcal E^-_{\rm FI}(\beta,F)$, $\mathcal E^+_{\rm FI}(\beta,F)$, $E^-_{\rm FI}(\beta,F)$, $E^+_{\rm FI}(\beta,F)$, $E_{\rm FI}(\beta,F)$. 

\subsection{From Fourier series to Fourier integrals}
\begin{lemma}\label{lem:fstofi}
Let $f(z)=\sum_{k\geq 0}a_k z^k\in H^2(\mathbb D)$ and define $F\in L^2([0,+\infty))$ by $F(t)=a_k$ if $t\in [k,k+1)$. Let $\beta>0$. There exists $N\in\NN$ such that, for any $n\geq N$, for any $t\in[-\pi,\pi)$ such that $|S_n f(e^{it})|\geq n^\beta$, then 
$$\left|\int_0^R F(u)e^{itu}du\right|\geq \frac12R^\beta,\qquad \forall R\in[n,n+1).$$
\end{lemma}

\begin{proof}
It is clear that $F\in L^2([0,+\infty))$ and that $\Vert F\Vert_{L^2([0,+\infty))}=\Vert f\Vert_{\mathcal H^2}$. 

If $R\in[n,n+1)$, we may write
$$\int_0^R F(u)e^{itu}du=\sum_{k=0}^n a_k\int_k^{k+1}e^{itu}du-a_n\int_R^{n+1}e^{itu}du.$$
The last term tends to $0$ uniformly in $t\in[-\pi,\pi)$. Computing the integral, we get
\begin{eqnarray*}
\int_0^R F(u)e^{itu}du&=&\sum_{k=0}^n a_k\frac{e^{i(k+1)t}-e^{ikt}}{it}+o(1)\\
&=&S_nf(e^{it})\frac{e^{it}-1}{it}+o(1).
\end{eqnarray*}
We conclude by observing that, for any $x\in [-\pi/2,\pi/2)$, $\left|\frac{\sin x}x\right|\geq \frac2\pi>\frac12$.
\end{proof}
This lemma will be used in the following way. Modulo the natural identification of $\TT$ with $[-\pi,\pi)$, for any $\beta>0$, $\mathcal E^-_{\rm FS}(\beta,f)\subset \mathcal E^-_{\rm FI}(\beta,F)$ and $\mathcal E_{\rm FS}^+(\beta,f)\subset\mathcal E_{\rm FI}^+(\beta,F)$. In particular, the construction of multifractal functions for Fourier series will help us to construct multifractal functions for Fourier integrals.

\subsection{From Dirichlet series to Fourier integrals}
\begin{lemma}\label{lem:dstofi}
Let $g(s)=\sum_{k\geq 1}b_kk^{-s}\in\mathcal H^2$ and let us define $G\in L^2([0,+\infty))$ by
$G(t)=\frac{b_k}{\big(\log(k+1)-\log(k)\big)^{1/2}}$ provided $t\in\left[\log k,\log(k+1)\right)$. 

Let $\beta>0$ and $A>0$. There exists $N\in\mathbb N$ such that, for any $n\geq N$, for any $t\in [-A,A]$ such that
$\left|\sum_{k=1}^n b_k k^{-\frac12+it}\right|\geq (\log n)^\beta$, then
$$\left|\int_0^R G(u)e^{itu}du\right|\geq\frac12R^{\beta},\qquad\forall R\in \left[\log n,\log(n+1)\right).$$
\end{lemma}

\begin{proof}
It is straightforward to check that $\|G\|_{L^2([0,+\infty))}=\|g\|_{\mathcal H^2}$. Moreover, as before, if $R\in\left[\log n,\log(n+1)\right)$,
\begin{eqnarray*}
\int_0^R G(u)e^{itu}du&=&\sum_{k=1}^n \frac{b_k}{\big(\log(k+1)-\log(k)\big)^{1/2}}\frac{e^{it\log(k+1)}-e^{it\log k}}{it}+o(1)\\
&=&\sum_{k=1}^n b_k(\log(1+1/k))^{1/2}k^{it}\times\frac{e^{it\log(1+1/k)}-1}{it\log(1+1/k)}+o(1).
\end{eqnarray*}
Taylor's formula ensures that for any real number $u$,
$$\left\vert e^{iu}-1-iu\right\vert=\left\vert-u^2\int_0^1e^{isu}(1-s)ds\right\vert\le\frac{\vert u\vert^2}2.$$
It follows that
\begin{eqnarray*}
\left\vert\int_0^R G(u)e^{itu}du-\sum_{k=1}^n b_k\left(\log(1+1/k)\right)^{1/2}k^{it}\right\vert&\le&o(1)+\frac{\vert t\vert}2\sum_{k=1}^n\vert b_k\vert\left(\log(1+1/k)\right)^{3/2}\\
&\le&o(1)+\frac{\vert t\vert}2\sum_{k=1}^n\frac{|b_k|}{k^{3/2}}.
\end{eqnarray*}
On the other hand, 
$$\left\vert(\log(1+1/k))^{1/2}-k^{-1/2}\right\vert=\frac{\left\vert\log(1+1/k)-k^{-1}\right\vert}{(\log(1+1/k))^{1/2}+k^{-1/2}}\ll k^{-3/2}.$$
Using Cauchy-Schwarz inequality, we can conclude that
$$\left\vert \int_0^R G(u)e^{itu}du-\sum_{k=1}^nb_kk^{-\frac12+it}\right\vert\ll o(1)+(1+\vert t\vert)\Vert g\Vert_{\mathcal H^2}$$
and the lemma follows easily. 
\end{proof}
\begin{remark}
Going slightly further, we can give a new proof of a very important statement in the theory of Hardy spaces of Dirichlet series: $\mathcal H^2$ may be embedded into the invariant Hardy space $H^2_i(\mathbb C_{1/2})$. Precisely, assume that $g$ is a Dirichlet polynomial. Then we have shown that 
$$\left|\int_0^{+\infty}G(u)e^{itu}du-g\left(\frac12-it\right)\right|\ll (1+|t|)\|g\|_{\mathcal H^2}.$$
Therefore, by Plancherel's theorem,
\begin{eqnarray}
\nonumber
\int_0^1 \left|g\left(\frac 12-it\right)\right|^2dt&\ll&\|g\|_{\mathcal H^2}^2+\int_0^1 \left|\int_0^{+\infty}G(u)e^{itu}du\right|^2dt\\
&\ll&\|g\|_{\mathcal H^2}^2+\|G\|_2^2\nonumber \\
&\ll&\|g\|_{\mathcal H^2}^2. \label{eq:embedding}
\end{eqnarray}
\end{remark}

\subsection{From Fourier integrals to Dirichlet series}

We will also need to go in the reverse direction.
\begin{lemma}\label{lem:fitods}
Let $G\in L^2([0,+\infty))$ and define $g(s)=\sum_{k\geq 1}b_kk^{-s}\in\mathcal H^2$ by
$$b_k=\int_{\log k}^{\log(k+1)}G(u)e^{u/2}du.$$
Let $\beta>0$ and $A>0$. There exists $N\in\mathbb N$ such that, for any $n\geq N$, for any $t\in[-A,A]$, for any $R\in \left[\log(n),\log(n+1)\right)$ such that $\left|\int_0^RG(u)e^{itu}du\right|\geq R^\beta,$
then 
$$\left|\sum_{k=1}^n b_k k^{-\frac12+it}\right|\geq\frac12(\log n)^{\beta}.$$
\end{lemma}
\begin{proof}
We first show that $g$ belongs to $\mathcal H^2$. We just write
\begin{eqnarray*}
\sum_{k=1}^{+\infty}|b_k|^2&=&\sum_{k=1}^{+\infty}\left|\int_{\log k}^{\log(k+1)}G(u)e^{u/2}du\right|^2\\
&\leq&\sum_{k=1}^{+\infty}\int_{\log k}^{\log(k+1)}|G(u)|^2du\int_{\log k}^{\log(k+1)}e^udu=\|G\|_{L^2([0,+\infty))}^2.
\end{eqnarray*}
Moreover, and proceeding as above, if $R\in\left[\log n,\log(n+1)\right)$,
\begin{eqnarray*}
&&\left|\int_0^R G(u)e^{itu}du-\sum_{k=1}^n b_kk^{-\frac12+it}\right|\\
&&\leq o(1)+\sum_{k=1}^n \int_{\log k}^{\log(k+1)}\left|G(u)e^{itu}-G(u)\frac{e^{u/2}}{k^{1/2}}e^{it\log k}\right|du\\
&&= o(1)+\sum_{k=1}^n \int_{\log k}^{\log(k+1)}|G(u)e^{u/2}|\times\left|e^{\left(it-\frac12\right)u}-e^{\left(it-\frac12\right)\log k}\right|du\\
\end{eqnarray*}
Observe that if $\log k\le u\le \log(k+1)$, 
\begin{eqnarray*}
\left\vert e^{\left(it-\frac12\right)u}-e^{\left(it-\frac12\right)\log k}\right\vert&=&\left\vert\int_{\log k}^u\left(it-\frac12\right)\times e^{\left(it-\frac12\right)s}ds\right\vert\\
&\le&\left(\log(k+1)-\log k\right)\times\left\vert it-\frac12\right\vert\\
&\le&\frac1k\left(\vert t\vert +\frac12\right).
\end{eqnarray*}
If follows that
\begin{eqnarray*}
\left|\int_0^R G(u)e^{itu}du-\sum_{k=1}^n b_kk^{-\frac12+it}\right|&\ll&
 o(1)+(|t|+1)\sum_{k=1}^n \int_{\log k}^{\log(k+1)}\frac{|G(u)e^{u/2}|}{k}du\\
&\ll& o(1)+(|t|+1)\left(\sum_{k=1}^n \left(\int_{\log k}^{\log(k+1)}|G(u)e^{u/2}|du\right)^2\right)^{1/2}\\
&\ll& o(1)+(|t|+1)\|G\|_{L^2([0,+\infty))}
\end{eqnarray*}
and the conclusion of the lemma follows easily.
\end{proof}

\subsection{A localization lemma for Fourier integrals}
The three above lemmas are useful to transfer multifractal functions from one situation to another one. We need a last tool to bound the dimensions. It is the exact analogue of Lemma \ref{lem:fourier1} for Fourier integrals.
For a function $F\in L^2(\RR)$, we denote by $\hat F$ or by $\mathcal F(F)$ its Fourier transform $\mathcal F(F)(\xi)=\int F(t)e^{it\xi}dt$, and by $\overline{\mathcal F}(F)$ its conjugate Fourier transform, so that $\overline{\mathcal F}\mathcal F(F)=2\pi F$. We also denote by $F_R$ the function $\mathbf 1_{[-R,R]}F$.

\begin{lemma}\label{lem:boundfi}
Let $G\in L^2(\RR)$, $R\ge 2$, $t\in\RR$ such that $|\widehat {G_R}(t)|\geq \|G\|_2$ and let $I$ be the interval of size $1/R$ centred at $t$. Then 
$$\|\widehat {G_R}\|_{L^2(I)}\ge\delta \frac{R^{-1/2}}{\log R}|\widehat {G_R}(t)|$$
for some constant $\delta$ independent of $R$, $t$ and $G$.
\end{lemma}
\begin{proof}
Let $w\in\mathcal C^{\infty}(\RR)$ be a positive function with support in $[-1,1]$ such that $0\leq w\leq 1$, $w(0)=1$ and such that there exist two positive constants $A$ and $B$ such that, for all $\xi\in\mathbb R$, 
$$|\widehat w(\xi)|\leq Ae^{-B|\xi|^{1/2}}$$
(for the construction of such a function, we refer for instance to \cite[Lemma 6]{Aub06}). We set $w_I(x)=w\left(\frac{2(x-t)}{|I|}\right)$ which has support in $I$ and satisfies 
$$|\widehat {w_{I}}(\xi)|\leq C |I|e^{-D |\xi|^{1/2}|I|^{1/2}}.$$
Let $\rho=R+\gamma R(\log R)^2$ for some large $\gamma>0$ and define $f_1,f_2\in L^2(\RR)$ by 
\begin{eqnarray*}
\overline{\mathcal F}(f_1)&=&\mathbf 1_{[-\rho,\rho]}\overline{\mathcal F}(\widehat{G_R}w_I)\\
f_2&=&\widehat{G_R}w_I-f_1.
\end{eqnarray*}
By Nikolskii's inequality for Fourier integrals (see \cite{NW78}), 
\begin{eqnarray}
\|f_1\|_\infty&\ll&\rho^{1/2}\|f_1\|_2 \nonumber\\
&\ll&\rho^{1/2}\|\widehat{G_R}w_I\|_2 \nonumber\\
&\ll&\rho^{1/2}\|\widehat{G_R}\|_{L^2(I)} \label{eq:nikolskii}.
\end{eqnarray}
On the other hand, we may write
\begin{eqnarray*}
\overline{\mathcal F}(f_2)&=&\big(\mathbf 1_{(-\infty,-\rho)}+\mathbf 1_{(\rho,+\infty)}\big)G_R\star\overline{\mathcal F}(w_I)
\end{eqnarray*}
Now,
\begin{eqnarray*}
\int_{\rho}^{+\infty}|G_R\star\overline{\mathcal F}(w_I)(\xi)|d\xi&\leq&\int_{\rho}^{+\infty}\left|\int_{-R}^R G(u) \overline{\mathcal F}(w_I)(\xi-u)du\right|d\xi\\
&\leq&\int_{\rho}^{+\infty}\left(\int_{-R}^R |G(u)|^2du\right)^{1/2}\left(\int_{-R}^R | \overline{\mathcal F}(w_I)(\xi-u)|^2du\right)^{1/2}d\xi\\
&\ll&\|G_R\|_2 R^{-1/2}\int_{\rho}^{+\infty}\exp\big(-D |\xi-R|^{1/2}R^{-1/2}\big)d\xi\\
&\ll&\|G_R\|_2 R^{1/2}\int_{\frac{\rho-R}R}^{+\infty}\exp\big(-Du^{1/2}\big)du\\
&=&\|G_R\|_2 R^{1/2}\int_{\gamma^{1/2}\log R}^{+\infty}\exp\left(-Dv\right)2vdv\\
&\ll&\|G_R\|_2 R^{1/2}\exp\left(-\frac{D}2\gamma^{1/2}\log R\right)\\
&\leq&\frac{\|G\|_2}2
\end{eqnarray*}
provided $\gamma$ is large enough. The same is true for the integral $\int_{-\infty}^{-\rho}$. It follows that
$\left\Vert \overline{\mathcal F}(f_2)\right\Vert_1\le\Vert G\Vert_2$. Hence,
\begin{equation}
\|f_2\|_\infty\leq\frac1{2\pi} \left\Vert\overline{\mathcal F}(f_2)\right\Vert_1\leq\frac{\|G\|_2}2\leq \frac{|\widehat{G_R}(t)|}{2}. \label{eq:nikolskii2}
\end{equation}
Since $f_1(t)+f_2(t)=\widehat {G_R}w_I(t)$, we deduce from \eqref{eq:nikolskii} and \eqref{eq:nikolskii2} that
$$\|\widehat {G_R}\|_{L^2(I)}\gg \rho^{-1/2}\|f_1\|_\infty\ge\rho^{-1/2}\big(|\widehat {G_R}(t)|-\|f_2\|_\infty\big)\ge\frac12 \rho^{-1/2}|\widehat {G_R}(t)|\gg\frac{R^{-1/2}}{\log R}|\widehat {G_R}(t)|.$$
\end{proof}

\subsection{Proof of Theorem \ref{thm:maindirichlet}}
We first work with Fourier integrals and we intend to show that for any $F\in L^2([0,+\infty))$ and any $\beta\in[0,1/2]$,
$$\dimp\big(\mathcal E^+_{\rm FI}(\beta,F)\big)\leq 1-2\beta\quad\mbox{and}\quad\dimh\big(\mathcal E_{\rm FI}^-(\beta,F)\big)\leq 1-2\beta$$  
We can suppose $\beta>0$ and it suffices to work with $\mathcal E^+_{\rm FI}(\beta,F)\cap [0,1]$ and $\mathcal E^-_{\rm FI}(\beta,F)\cap [0,1]$. For $\lambda$ a dyadic interval of the $j$-th generation, we define
$$e_\lambda(F)=\left(\int_{3\lambda}|\widehat{ F_{2^j}}(\xi)|^2d\xi\right)^{1/2}.$$
By Lemma \ref{lem:boundfi}, $\mathcal E^+_{\rm FI}(\beta,F)\subset \mathcal F^+(\frac12-\beta,F)$. Since the sequence $(e_\lambda(F))_{\lambda\in\Lambda_j}$ satisfies 
$$\sum_{\lambda\in\Lambda_j} |e_\lambda(F)|^2\ll\left\Vert\widehat{F_{2^j}}\right\Vert_2^2\ll \|F\|_2^2,$$ we deduce the result on the packing dimension of $\mathcal E^+_{\rm FI}(\beta,F)\cap [0,1]$ from Proposition \ref{prop:bounddimension}. The proof for the Hausdorff dimension needs a more sophisticated tool. Define now 
$$e_\lambda(F)=\sup_{2^j\leq R<2^{j+1}}\left(\int_{3\lambda}|\widehat {F_R}(\xi)|^2d\xi\right)^{1/2}.$$
Let $\beta\in(0,1/2]$ and  $t\in \mathcal E^-_{\rm FI}(\beta,F)\cap [0,1]$. For any $\veps>0$, there exists $R$ as large as we want such that $|\widehat {F_R}(t)|\geq R^{\beta-\veps}$. Let $j\geq 1$ such that $2^j\leq R<2^{j+1}$. Then $3I_j(t)$ contains the interval with center $t$ and radius $R^{-1}$. This implies, by Lemma \ref{lem:boundfi}, 
$$e_j(F,t)\gg \frac{R^{-1/2}}{\log R}R^{\beta-\veps}\gg 2^{-j\left(\frac12-\beta+2\veps\right)}.$$
Therefore, $ \mathcal E^-_{\rm FI}(\beta,F)\cap [0,1]\subset \mathcal F^-\big(\frac12-\beta+2\veps,F)$. To apply Proposition \ref{prop:bounddimension}, we need to prove the inequality $\sum_{\lambda\in\Lambda_j} |e_\lambda(F)|^2\ll \|F\|_2^2$ which can be obtained by writing
$$\sum_{\lambda\in\Lambda_j}|e_\lambda(F)|^2=\sum_{\lambda\in\Lambda_j}\sup_{2^j\leq R<2^{j+1}}\left(\int_{3\lambda}|\widehat {F_R}(\xi)|^2d\xi\right)\ll \int_\RR\sup_{2^j\leq R<2^{j+1}}|\widehat {F_R}(\xi)|^2d\xi\ll\Vert F\Vert_2^2.$$
The last inequality is a consequence  of the Carleson-Hunt theorem for Fourier integrals (see \cite[Chapter 6]{Gramodern}).

\bigskip

We now come back to Dirichlet series and to the proof of Theorem \ref{thm:maindirichlet}. Let us begin with point (i). Let $g\in\mathcal H^2$ and let $G$ be associated to $g$ by Lemma \ref{lem:dstofi}. If $\beta>0$, the lemma states
$$ \mathcal E^-_{\rm DS}(\beta,g)\subset \mathcal E^-_{\rm FI}(\beta,G)\quad\textrm{ and }\quad \mathcal E^+_{\rm DS}(\beta,g)\subset \mathcal E^+_{\rm FI}(\beta,G)$$
and the previous discussion gives the aforementioned bound on the dimension of $ \mathcal E^-_{\rm DS}(\beta,g)$ and $\mathcal E^+_{\rm DS}(\beta,g)$. Now again, the case $\beta=0$ which is not included in Lemma \ref{lem:dstofi} is obvious.

\vskip 0.2cm

Let us turn to point (iii) and to the construction of multifractal Dirichlet series. We begin with the multifractal function $f\in L^2(\TT)$ given by Theorem \ref{thm:mainfourier}. Applying successively Lemma \ref{lem:fstofi} and Lemma \ref{lem:fitods}, we get $g\in\mathcal H^2$ such that, for all $\beta\in(0,1/2)$, 
$$\mathcal H^{\psi_\beta}\left(\mathcal E^+_{\rm DS}(\beta,g)\right)>0$$
for some gauge function $\psi_\beta$ satisfying $\psi_\beta(x)=o(x^s)$ for any $s<1-2\beta$ (see the details of the proof of Theorem \ref{thm:mainfourier}).
We end up the proof as for Theorem \ref{thm:mainpacking} to conclude that if $0<\beta<1/2$,
$$\dimh\big( E_{\rm DS}(\beta,g)\big)=\dimp\big( E_{\rm DS}(\beta,g)\big)= 1-2\beta.$$ 
The key point is that $\dimh\left(\mathcal E^-_{\rm DS}(\alpha,g)\right)\leq 1-2\alpha$ for $\alpha>\beta$. 

The values $\beta=0$ and $\beta=1/2$ need a different argument. The case $\beta=1/2$ is an obvious consequence of point (i) of the theorem. For the case $\beta=0$, we can mention a version of Carleson's convergence theorem which says that the series $\sum_{k\ge 1}a_kk^{-\frac12+it}$ is almost surely convergent (see for example \cite{HS} or \cite{KONQUEFF}).

\vskip 0.2cm

It remains to prove point (ii) of Theorem \ref{thm:maindirichlet} and to show that the set $\mathcal R$ of functions $g\in\mathcal H^2$ such that for all $\beta\in[0,1/2]$, $\dimh\left(E^-_{\rm DS}(\beta,g)\right)=1-2\beta$, is residual. We deduce this from the existence of at least one function in $\mathcal R$ (namely the function given by point (iii)), the density of Dirichlet polynomials, and a routine argument mimicking that of Theorem \ref{thm:mainhausdorff}. $\qed$

\subsection{An open question}
Beyond $\mathcal H^2$, a theory of $\mathcal H^p$-spaces of Dirichlet series is under development. For $p\geq 1$, $\mathcal H^p$ may be defined as the closure of the set of Dirichlet polynomials for the norm
$$\|P\|_p^p=\lim_{T\to+\infty}\frac 1T\int_0^T |P(it)|^p dt.$$
It can be proved that, for any $f(s)=\sum_{k\ge 1} a_k k^{-s}\in\mathcal H^p$, for any $t\in\RR$ and for any $n\geq 2$, 
$$\left|\sum_{k=1}^n a_kk^{-\frac 12+it}\right|\ll \|f\|_p(\log n)^{1/p}.$$
Do we have results similar to Theorem \ref{thm:maindirichlet} for $\mathcal H^p$? An obvious obstruction is that we do not know whether $\mathcal H^p$ embeds into $H^p_i(\mathbb C_{1/2})$ (see \eqref{eq:embedding}) when $p$ is not an even integer.
\providecommand{\bysame}{\leavevmode\hbox to3em{\hrulefill}\thinspace}
\providecommand{\MR}{\relax\ifhmode\unskip\space\fi MR }
\providecommand{\MRhref}[2]{%
  \href{http://www.ams.org/mathscinet-getitem?mr=#1}{#2}
}
\providecommand{\href}[2]{#2}


\end{document}